\documentclass[11pt, reqno]{amsart}
\usepackage{amsmath}
\usepackage{amsfonts}
\usepackage{amsthm}
\usepackage{latexsym}
\usepackage{amssymb}
\usepackage{enumerate}
\usepackage{mathrsfs}
\usepackage{color}
\usepackage{vmargin}
\usepackage{graphicx}  
\usepackage{esvect}
\usepackage{esint}
\usepackage{leftindex}
\usepackage[utf8]{inputenc}
\usepackage[T1]{fontenc}

\makeatletter
\newcommand{\vast}{\bBigg@{4}}
\newcommand{\Vast}{\bBigg@{5}}
\newcommand{\biggg}{\bBigg@{3}}
\newcommand{\Biggg}{\bBigg@{3.5}}
\makeatother

\def\XXint#1#2#3{{\setbox0=\hbox{$#1{#2#3}{\int}$ }
\vcenter{\hbox{$#2#3$ }}\kern-.6\wd0}}

\DeclareMathOperator*{\diam}{diam}

\newcommand*{\dd}{\mathop{}\!\mathrm{d}}
\newcommand{\vertiii}[1]{{\left\vert\kern-0.25ex\left\vert\kern-0.25ex\left\vert #1 
    \right\vert\kern-0.25ex\right\vert\kern-0.25ex\right\vert}}

\newcommand{\mres}{\mathbin{\vrule height 1.6ex depth 0pt width
0.13ex\vrule height 0.13ex depth 0pt width 1.3ex}}
\long\def\symbolfootnote[#1]#2{\begingroup%
\def\thefootnote{\fnsymbol{footnote}}\footnote[#1]{#2}\endgroup}

\setmarginsrb{20mm}{20mm}{20mm}{20mm}{10mm}{10mm}{10mm}{10mm}

\newtheoremstyle{remark}
  {}{}{}{}{\bfseries}{.}{.5em}{{\thmname{#1 }}{\thmnumber{#2}}{\thmnote{ (#3)}}}
\theoremstyle{remboldstyle}

\RequirePackage{amsthm}

\newtheorem{tw}{Theorem}[section]

\newtheorem{lem}[tw]{Lemma}

\newtheorem{remark}[tw]{Remark}
\newtheorem{ex}[tw]{Example}
\newtheorem{q}{Question}

\def\={\hspace{-3mm}&=&\hspace{-3mm}}

\usepackage{tikz}
\usetikzlibrary{arrows,calc,patterns,cd}
\usepackage{wrapfig}

\let\oldmarginpar\marginpar
\renewcommand{\marginpar}[2][rectangle,draw,text width= 1.5cm,rounded corners]{
    \oldmarginpar{
    \scriptsize \tikz \node at (0,0) [#1]{#2};}
    }
\begin{document}

\date{}

\title[Optimal Embeddings of Campanato and Morrey spaces]{\bf Optimal Embeddings of Campanato and Morrey spaces}

\author{Przemys{\l}aw G\'{o}rka \& Zuzanna St\k{e}pie\'{n}}

\maketitle
\begin{abstract}
We investigate Campanato and Morrey spaces on metric-measure spaces and establish their relationships with H\"{o}lder spaces. In particular, we show that the embedding of a Campanato space into a Hölder space is equivalent to the existence of a lower bound for the underlying measure. Several examples are provided to illustrate the results.
\end{abstract}
\bigskip

\noindent
{\bf Keywords}: Campanato spaces, Morrey spaces, metric-measure spaces, lower Ahlfors regular measure
\medskip

\noindent
\emph{Mathematics Subject Classification (2020):} 46E35; 30L99

\section{Introduction}
The Morrey and Campanato spaces play a fundamental role in harmonic analysis and partial differential equations (see, e.g., \cite{C, Gi, KJF, LR, P, YYZ}). Various definitions of these spaces have been introduced and studied; see for instance the work of Spanne \cite{Sp} and Nakai \cite{N}. For an overview of Morrey and Campanato spaces, we refer the reader to \cite{Adams, RSS}. In the present paper we investigate these spaces in the setting of metric-measure spaces, where alternative definitions have already been considered (see, e.g., \cite{G2}).

The primary objective of this work is to study Campanato and Morrey spaces on metric-measure spaces. We establish connections between the embedding of Campanato spaces into H\"{o}lder spaces on metric-measure spaces and the lower regularity of the measure. The necessary condition for the embedding of Campanato spaces into H\"{o}lder spaces in Euclidean domains was previously obtained by Op\v{e}la \cite{O}. We emphasize that the role of measure regularity in embedding results between function spaces on metric-measure spaces has been examined before. For instance, in the case of Sobolev spaces, Alvarado-G\'{o}rka-Haj{\l}asz \cite{Alvarado} described the relation between embeddings of the Haj{\l}asz–Sobolev space $M^{1,p}$ and the lower regularity of the measure (see also \cite{Gorka} for partial results, \cite{AYY, Karak1, Karak2} for results concerning Triebel–Lizorkin and Besov spaces, and \cite{GS} for fractional Sobolev spaces).

The remainder of this article is organized as follows. In Section~2 we recall the basic terminology concerning metric-measure spaces and introduce the main classes of function spaces considered in this paper, including the Campanato, Morrey, and H\"{o}lder spaces. We also present several technical lemmas that will play a central role in the proofs of the main theorems. Section~3 contains our principal results. In Theorems \ref{thm:embMorrCampLeb} and \ref{thm:embCampMorr} we investigate the relations between Morrey and Campanato spaces. The connections between Campanato and H\"{o}lder spaces are established in Theorems \ref{camp} and \ref{dolne}, and summarized in Theorems \ref{ost1} and \ref{ost2}. Many of the results are illustrated with examples, and we conclude with some open problems.

\section{Preliminaries}

\subsection{Metric-measure spaces}
We will always assume that metric spaces have at least two points. 
  We say that a metric space \(\left(X,d\right)\) is \texttt{$\eta$-uniformly perfect}, where $\eta \in (0,1)$ if for each $x \in X$ and each $r \in (0,\infty)$ one has
   \[ 
B\left(x,r\right)\setminus B\left(x,\eta r\right)\neq \emptyset, \,\text{whenever} \,X\setminus B\left(x,r\right)\neq \emptyset.
\]
Let us note that the class of uniformly perfect spaces contains connected spaces. Furthermore, this class contains disconnected spaces, e.g. the Cantor set.
\begin{lem}\label{lem:goodBall}
    Let \(\left(X,d\right)\) be $\eta$-uniformly perfect metric space, where $ \eta \in (0,1)$. Let \(x\in X, r \in (0,\infty)\) be such that \(X\setminus B\left(x,r\right)\neq \emptyset\). Then, for every $x_0 \in X$ there exists \(\Tilde{x}\in X\) such that \(B\left(\Tilde{x}, \Tilde{r}\right)\subset B\left(x,r\right)\) and \(x_0\notin B\left(\Tilde{x}, 4\Tilde{r}\right)\), where \(\Tilde{r}=\eta r/ 20\).
\end{lem}
\begin{proof}
    From uniform perfectness, there exists \(x_1\in B\left(x,r/2\right)\setminus B\left(x,\eta r/2\right)\). Thus, \(d\left(x,x_0\right)\geq \eta r/4\) or \(d\left(x_1,x_0\right)\geq \eta r/4\) holds. Indeed, if that were not the case, we would get
    \[
d\left(x,x_1\right)\leq d\left(x,x_0\right)+d\left(x_0,x_1\right)<\eta r/4+\eta r/4=\eta r/2,
\]
    which leads to a contradiction with the choice of \(x_1\).

    Case I: \(d\left(x,x_0\right)\geq \eta r/4.\) Put \(\Tilde{x}=x\) and \(\Tilde{r}=\eta r/20\). Then for \(y\in B\left(\Tilde{x},\Tilde{r}\right)\), we get
    \[d\left(x,y\right)=d\left(\Tilde{x},y\right)<\eta r/20<r.\]
    Furthermore
    \[d\left(\tilde{x},x_0\right)=d\left(x,x_0\right)\geq \eta r/4>4\eta r/20=4\Tilde{r}.\]
    Therefore  \(B\left(\Tilde{x}, \Tilde{r}\right)\subset B\left(x,r\right)\) and \(x_0\notin B\left(\Tilde{x}, 4\Tilde{r}\right)\).

    Case II: \(d\left(x_1,x_0\right)\geq \eta r/4.\) Put \(\Tilde{x}=x_1\) and \(\Tilde{r}= \eta r/ 20\). 
	Then for  \(y\in B\left(\Tilde{x},\Tilde{r}\right)\), we have
    \[d\left(x,y\right)\leq d\left(x,x_1\right)+d\left(\tilde{x},y\right)<r/2+\eta r/20<r.\]
    Moreover
    \[d\left(\tilde{x},x_0\right)=d\left(x_1,x_0\right)\geq \eta r/4>4\eta r/20=4\Tilde{r}.\]
    Hence \(B\left(\Tilde{x}, \Tilde{r}\right)\subset B\left(x,r\right)\) and \(x_0\notin B\left(\Tilde{x}, 4\Tilde{r}\right)\).
\end{proof}
The triplet $(X, d, \mu)$ will be called a \texttt{metric-measure space} if $(X, d)$ is a metric space and $\mu$ is a Borel measure on $X$ such that $\mu(B(x,r)) \in (0, \infty)$ for any $x \in X$ and any $r \in (0, \infty)$. It is well known (see \cite{DG, G}) that metric-measure spaces are separable.

Given a metric-measure space $(X, d, \mu)$, the measure $\mu$ is said to be \texttt{doubling} provided there exists a constant $C_{\mu} \in (0,\infty)$ such that
\[
\mu(2B) \leq C_{\mu} \mu(B)\, \text{for all balls}\, B \subset  X.
\]

Let  $(X, d, \mu)$ be a metric-measure space and $Q>0$. We shall say that the measure $\mu$ is \texttt{lower $Q$-regular} if there is a constant $C_Q>0$ such that 
\begin{equation*} \label{regular}
\mu(B(z,r)) \geq C_Q^{-1} r^Q \mbox{ for } r \in (0, 3 \diam X), \,\, z\in X.
\end{equation*}
We shall say that the measure $\mu$ is \texttt{upper $Q$-regular} if there is a constant $C_Q>0$ such that 
\begin{equation*} \label{regular}
\mu(B(z,r)) \leq C_Q r^Q \mbox{ for } r \in (0, 3 \diam X), \,\, z\in X.
\end{equation*}
Furthermore, we shall say that the measure $\mu$ is \texttt{Ahlfors $Q$-regular} if there is a constant $C_Q\geq 1$ such that 
\begin{equation*} \label{regular}
C_Q^{-1} r^Q \leq \mu(B(z,r)) \leq C_Q r^Q \mbox{ for } r \in (0, 3 \diam X), \,\, z\in X.
\end{equation*}
It is easy to convince oneself that if  $\mu$ is Ahlfors $Q$-regular measure on a metric-measure space $(X, d, \mu)$, then $(X, d)$ is $\eta$-uniformly perfect for every $\eta \in \left(0, C_Q^{-2/Q}\right)$ and $\mu$ is doubling with $C_{\mu}=C_Q^2 2^Q$.

Recall that a Borel measure $\mu$ is said to be Borel regular, if for every measurable set $E$, there exists a Borel set $B$ such that $E \subset B$ and $\mu(B)=\mu(E)$.  
Let us stress that in \cite{AGS} the characterizations of the Lusin's theorem and the existence of Borel representatives are given via the regularity properties of the measure. On the other hand, 
in light of the following example, we see that Ahlfors regular measure need not be necessarily Borel regular. 
\begin{ex} \label{Paul}
	There exists a measure space $(\mathbb{R}^n, \mathcal{F}, \nu)$ such that\footnote{By $\lambda$ we denote the Lebesgue measure on $\mathbb{R}^n$ and by $\mathcal{L}_n$ we denote the $\sigma$-algebra of Lebesgue measurable subsets of $\mathbb{R}^n$.} $\mathcal{L}_n \subset \mathcal{F}$, $\nu|_{\mathcal{L}_n}=\lambda$ and $\nu$ is not Borel regular.
\end{ex}
\begin{proof}
Let $V\subset [0,1]^n$ be a non Lebesgue measurable set such that the family \(\left\{V  +q(1,0,...,0)\right\}_{q\in\mathbb{Q}\cap\left(-1,1\right)}\) is pairwise disjoint. Since for $q\in\mathbb{Q}\cap\left(-1,1\right)$ we have $V  +q(1,0,...,0) \subset [-1,2]\times [0,1]^{n-1}$, inner measure is superadditive and the inner Lebesgue measure $\lambda_*$ is translation invariant, we get $\lambda_*(V)=0$. Let $C=\mathbb{R}^n \setminus V$ and let us define the outside measure $\nu$ on $\mathbb{R}^n$ as follows
\[
	\nu(F) = \lambda^* (C \cap F) \, \text{for} \, F \subset \mathbb{R}^n.
\]
One can easily see that\footnote{By $ \mathcal{M}(\nu)$ we denote the Carath\'{e}odory $\sigma$-algebra.} $ \mathcal{L}_n \subset  \mathcal{M}(\nu):=\mathcal{F} $. Moreover, since $\lambda_*(\mathbb{R}^n \setminus C)=0$, by \cite{S} we have
\[
	\nu(F) = \lambda^* (C \cap F)=\lambda (F) \, \text{for} \, F \in \mathcal{L}_n.
\]
Finally, since $C \notin \mathcal{L}_n$, we get that $\nu$ is not Borel regular (see Lemma 3.3.13 in \cite{HKST} for similar considerations).
\end{proof}

\begin{lem}\label{lem:auxLem}
        Let \(\left(X,d,\mu\right)\) be a metric-measure space such that \(\left(X,d\right)\) is compact and \(\eta\)-uniformly perfect for some $\eta \in (0,1)$. If the measure $\mu$ is not lower \(Q\)-regular, then there exist sequences \(\left(x_n\right)_{n=1}^{\infty}\subseteq X\), \(\left(r_n\right)_{n=1}^{\infty}\subseteq \left(0,3\diam X\right)\) such that
        \begin{enumerate}
            \item \label{it:a} the sequence \(\left(r_n\right)_{n=1}^{\infty}\) is decreasing and convergent to 0;
            \item \label{it:b} the sequence \(\left(x_n\right)_{n=1}^{\infty}\) converges to a point \(x_0\in X\);
            \item \label{it:c} for all \(n\in\mathbb{Z}_+\) holds that
            \[\mu\left(B\left(x_n,r_n\right)\right)\leq C10^{-n}r_n^Q,\]
            with \(C=20^Q\eta^{-Q};\)
            \item \label{it:d} for any \(m,n\in\mathbb{Z}_+\) such that \(m>n\) the following equality holds
            \[B\left(x_m,r_m\right)\cap B\left(x_n,3r_n\right)=\emptyset,\]
            in particular the family of balls \(\left(B\left(x_n,r_n\right)\right)_{n=1}^{\infty}\) is pairwise disjoint;
            \item \label{it:e} for all \(n\in\mathbb{Z}_+\) we have
            \[X\setminus B\left(x_n,r_n\right)\neq \emptyset.\]
        \end{enumerate}
    \end{lem}
    \begin{proof}
       Since the measure $\mu$ is not lower $Q$-regular, there exist \(\left(y_n\right)_{n=1}^{\infty}\subseteq X\) and \(\left(\rho_n\right)_{n=1}^{\infty}\subseteq\left(0, 3\diam X\right)\) 
such that for $n\in \mathbb{Z}_+$ one has
        	\[
		\mu\left(B\left(y_n,\rho_n\right)\right)<10^{-n}\rho_n^Q.
	\]
        From the compactness, we may assume that there exist \(x_0\in X, \rho_0\in\left[0,3\diam X\right]\) such that \(y_n\underset{n\to\infty}{\longrightarrow}x_0\), \(\rho_n\underset{n\to\infty}{\longrightarrow}\rho_0\). Let us observe that \(\rho_0=0\). Indeed, if $\rho_0 >0$, then we would have
        \[
		0<\mu\left(B\left(x_0,\rho_0\right)\right)\leq\liminf_{n\to\infty}\mu\left(B\left(y_n,\rho_n\right)\right)\leq\lim_{n\to\infty}10^{-n}\rho_n^Q=0,
	\]
        which is a contradiction\footnote{From the Fatou Lemma we have $\mu\left(B\left(x_0,(1-\epsilon)\rho_0\right)\right)\leq\liminf_{n\to\infty}\mu\left(B\left(y_n,\rho_n\right)\right)$ for any $\epsilon \in (0,1)$. Therefore, the inequality $\mu\left(B\left(x_0,\rho_0\right)\right)\leq\liminf_{n\to\infty}\mu\left(B\left(y_n,\rho_n\right)\right)$ holds.}. Furthermore, we may assume that the sequence \(\left(\rho_n\right)_{n=1}^{\infty}\) is decreasing and \(\rho_1<\left(10\mu\left(X\right)\right)^{\frac{1}{Q}}\). From that, for any \(n\in\mathbb{Z}_+\) we have 
        \[\mu\left(X\setminus B\left(y_n,\rho_n\right)\right)=\mu\left(X\right)-\mu\left(B\left(y_n,\rho_n\right)\right)>\mu\left(X\right)-10^{-n}\rho_n^Q\geq\mu\left(X\right)-10^{-1}\rho_1^Q>0.\]
        So \(X\setminus B\left(y_{n},\rho_{n}\right)\neq\emptyset\).

        Let \(n_1=1\). From Lemma \ref{lem:goodBall} we get \(x_1\in X, r_1>0\) such that \(B\left(x_1, r_1\right)\subseteq B\left(y_{n_1},\rho_{n_1}\right), r_1=\frac{\eta}{20}\rho_{n_1}\) and \(x_0\notin B\left(x_1,4r_1\right)\). Then we have that \(B\left(x_0,r_1\right)\cap B\left(x_1, 3r_1\right)=\emptyset\). 

        Let \(n_2\geq 2> n_1\) be such that for \(n\geq n_2\) the inclusion \(B\left(y_n, \rho_n\right)\subseteq B\left(x_0,r_1\right)\) holds. By Lemma \ref{lem:goodBall} we get \(x_2\in X, r_2>0\) such that \(B\left(x_2, r_2\right)\subseteq B\left(y_{n_2},\rho_{n_2}\right)\subseteq B\left(x_0, r_1\right), r_2=\frac{\eta}{20}\rho_{n_2}\), and \(x_0\notin B\left(x_2,4r_2\right)\). We see that \(B\left(x_0,r_2\right)\cap B\left(x_2, 3r_2\right)=\emptyset\). 

        In that way we construct inductively the sequences \(\left(x_l,r_l,n_l\right)_{l=1}^{\infty} \subset X \times (0,\infty)\times \mathbb{Z}_{+}\) such that $ \left(n_l\right)_{l=1}^{\infty}$ is increasing and for $l\geq 2$ we have
\begin{itemize}
	\item $B(x_l, r_l)\subset B(y_{n_l}, \rho_{n_l})\subset B(x_0, r_{l-1});$
	\item $r_l = \frac{\eta}{20}\rho_{n_l}, \, x_0 \notin B(x_l, 4r_l);$
	\item $ B\left(x_0,r_l\right)\cap B\left(x_l, 3r_l\right)=\emptyset.$
\end{itemize}
Now, having in mind the above properties, it is not difficult to see that $(x_n)_{n=1}^{\infty}$ and $(r_n)_{n=1}^{\infty}$ satisfy conditions $\eqref{it:a}-\eqref{it:e}$.
    \end{proof}

\subsection{Function spaces}
Here and in what follows, the integral average is denoted by
\[
f_A =\fint_A f d \mu=\frac{1}{\mu(A)}\int_A f d \mu,
\]
where $A\subset X$ is a measurable set of finite positive measure and $f$ is integrable over $A$.

    Let \(\left(X,d,\mu\right)\) be a metric-measure space, \(p\in\left[1,\infty\right)\), and \(\lambda\in\left[0,\infty\right)\). We denote by \(L^{p,\lambda}_M\left(X,\mu\right)\) a set of functions \(f\in L^p\left(X\right)\) such that
         \[\left\|f\right\|_{L^{p,\lambda}_M\left(X,\mu\right)}:=\sup_{x\in X, r>0}\left(\frac{1}{r^{\lambda}}\int_{B(x,r)}\left|f\right|^p\dd\mu\right)^{\frac{1}{p}}<\infty.\]
    The set \(L^{p,\lambda}_M\left(X,\mu\right)\) with a norm defined by \(\left\|\cdot\right\|_{L^{p,\lambda}_M\left(X,\mu\right)}\)is called \texttt{Morrey space}.

Furthermore, we denote by \(L^{p,\lambda}_C\left(X, \mu\right)\) a set of functions \(f\in L^p\left(X\right)\) such that
    \[\left[f\right]_{L^{p,\lambda}_C\left(X,\mu\right)}:=\sup_{x\in X, r>0}\left(\frac{1}{r^{\lambda}}\int_{B(x,r)}\left|f-f_{B(x,r)}\right|^p\dd\mu\right)^{\frac{1}{p}}<\infty.\]
    The set \(L^{p,\lambda}_C\left(X, \mu\right)\) with a norm defined by \(\left\|\cdot\right\|_{L^{p,\lambda}_C\left(X,\mu\right)}=\left\|\cdot\right\|_{L^p\left(X,\mu\right)}+\left[\cdot\right]_{L^{p,\lambda}_C\left(X,\mu\right)}\) is called \texttt{Campanato space}.

    Let \(\left(X,d\right)\) be a metric space and \(\alpha \in (0, \infty)\). We denote by \(C^{0,\alpha}\left(X\right)\) a set of continuous and bounded functions \(f:X\to \mathbb{R}\) such that
    \[\left[f\right]_{C^{0,\alpha}\left(X\right)}:=\sup_{x,y\in X, x\neq y}\frac{\left|f\left(x\right)-f\left(y\right)\right|}{d\left(x,y\right)^{\alpha}}<\infty.\]
    The set \(C^{0,\alpha}\left(X\right)\) with a norm defined by 
    \[\left\|f\right\|_{C^{0,\alpha}\left(X\right)}=\sup_{x\in X}\left|f\left(x\right)\right|+\sup_{x,y\in X, x\neq y}\frac{\left|f\left(x\right)-f\left(y\right)\right|}{d\left(x,y\right)^{\alpha}}\]
    is called \texttt{H\"{o}lder space}.

The Campanato and the H\"{o}lder space are Banach spaces. On the other hand, it is easy to see that \(\left\|\cdot\right\|_{L^{p,\lambda}_M\left(X,\mu\right)}\) is a norm and that for bounded metric spaces the Morrey space is a Banach space. 
\begin{ex}
    Let $((1,\infty), |.|, \lambda)$ and let $f_n : (1, \infty) \rightarrow \mathbb{R}$ be the following sequence $f_n (y) = 1/y^{1+1/n}$ and let \(f\left(y\right)=1/y\). Then for $x \in (1,\infty)$, $r \in (0,\infty)$ we have
    \begin{align*}
        \frac{1}{\sqrt{r}}\int_{\max\left\{x-r,1\right\}}^{x+r}\left|\frac{1}{y}-\frac{1}{y^{1+1/n}}\right|\dd y&=\frac{1}{\sqrt{r}}\int_{\max\left\{x-r,1\right\}}^{x+r}\int_{-1-1/n}^{-1}\ln yy^t\dd t\dd y\leq \frac{1}{\sqrt{r}}\int_{\max\left\{x-r,1\right\}}^{x+r}\frac{1}{n}\ln yy^{-1}\dd y\\
               &=\frac{1}{2n\sqrt{r}}\left(\left(\ln\left(x+r\right)\right)^2-\left(\ln\left(\max\left\{x-r,1\right\}\right)\right)^2\right).
    \end{align*}
Since the function \(\frac{1}{\sqrt{r}}\left(\left(\ln\left(x+r\right)\right)^2-\left(\ln\left(\max\left\{x-r,1\right\}\right)\right)^2\right)\) is bounded we see that 
\[\left\|f-f_n\right\|_{L^{p,\lambda}_M\left(X,\mu\right)}\xrightarrow[n\to\infty]{}0.\]
Therefore, since \(f\notin L^1\left(X,\mu\right)\), we get that $L^{1,1/2}_M\left(X,\mu\right)$ is not a Banach space.
\end{ex}

\begin{lem}\label{lem:equivOfNorms}
    Let \(\left(X,d, \mu\right)\) be a metric-measure space and let \(p\in\left[1,\infty\right), \lambda\in\left[0,\infty\right)\). For \(f\in L^p\left(X,\mu\right)\) let us define the following quantity
    \[\vertiii{f}_{p,\lambda}:=\sup_{x\in X,r>0}\left(\frac{1}{r^\lambda}\inf_{c\in\mathbb{R}}\int_{B\left(x,r\right)}\left|f-c\right|^p\dd\mu\right)^{\frac{1}{p}}.\]
Then
\begin{enumerate}
\item For \(f\in L^p\left(X,\mu\right)\) we have
    \[\vertiii{f}_{p,\lambda}\leq\left[f\right]_{L^{p,\lambda}_C\left(X,\mu\right)}\leq2\vertiii{f}_{p,\lambda}.\]
\item \(\vertiii{\cdot}_{p,\lambda}\) is a seminorm on \(L^{p,\lambda}_C\left(X,\mu\right)\).
\end{enumerate}
    \end{lem}

\subsection{Auxiliary results}
\begin{lem}\label{lem:average_rs}
    Let \(\left(X,d,\mu\right)\) be a metric-measure space and \(p\in \left[1,\infty\right), \lambda\in\left[0,\infty\right)\). Then for  \(f\in L^{p,\lambda}_C\left(X,\mu\right), x\in X\) and \(0<r<s\) we have 
\[
	\left|f_{B(x,r)}-f_{B(x,s)}\right|\leq 2^{1-\frac{1}{p}}\frac{\left(r^{\lambda}+s^{\lambda}\right)^{\frac{1}{p}}}{\left(\mu\left(B(x,r)\right)\right)^{\frac{1}{p}}}			\left[f\right]_{L^{p,\lambda}_C\left(X,\mu\right)}.\
\]
\end{lem}
\begin{proof}
	Since
    \begin{align*}
        \int_{B(x,r)}\left|f_{B(x,r)}-f_{B(x,s)}\right|^p\dd\mu\leq 2^{p-1}\left(\int_{B(x,r)}\left|f_{B(x,r)}-f\right|^p\dd\mu+\int_{B(x,r)}\left|f-f_{B(x,s)}\right|^p\dd\mu\right),
    \end{align*}
    we get
    \begin{align*}
        \mu\left(B(x,r)\right)\left|f_{B(x,r)}-f_{B(x,s)}\right|^p&\leq 2^{p-1}\left(\int_{B(x,r)}\left|f_{B(x,r)}-f\right|^p\dd\mu+\int_{B(x,r)}\left|f-f_{B(x,s)}\right|^p\dd\mu\right)\\
        &\leq2^{p-1}\left(\int_{B(x,r)}\left|f_{B(x,r)}-f\right|^p\dd\mu+\int_{B(x,s)}\left|f-f_{B(x,s)}\right|^p\dd\mu\right)\\
        &\leq2^{p-1}\left(r^{\lambda}+s^{\lambda}\right)\left[f\right]_{L^{p,\lambda}_C\left(X,\mu\right)}^p.
    \end{align*}
   Thus,
     \[\left|f_{B(x,r)}-f_{B(x,s)}\right|\leq 2^{1-\frac{1}{p}}\frac{\left(r^{\lambda}+s^{\lambda}\right)^{\frac{1}{p}}}{\left(\mu\left(B(x,r)\right)\right)^{\frac{1}{p}}}\left[f\right]_{L^{p,\lambda}_C\left(X,\mu\right)}.\]
\end{proof}
\begin{lem}\label{lem:averages_k}
    Let \(\left(X,d,\mu\right)\) be a metric-measure space such that the measure \(\mu\) is lower \(Q\)-regular, where $Q \in (0,\infty)$. Let \(p\in\left[1,\infty\right), \, \lambda\in\left[0,\infty\right)\setminus\left\{Q\right\}\). Then for \(f\in L^{p,\lambda}_C\left(X, \mu\right)\), \(x\in X\), \(r\in\left(0, 3\diam X\right)\), and \(k\in\mathbb{N}\) the following inequality holds
    \[
	\left|f_{B(x,r)}-f_{B(x,r/2^k)}\right|\leq Cr^{\frac{\lambda-Q}{p}}\left(2^{\frac{\left(Q-\lambda\right)k}{p}}-1\right)\left[f\right]_{L^{p,\lambda}_C\left(X,\mu\right)},
\]
    where \(C=2^{1+\frac{Q-1}{p}}\left(1+2^{-\lambda}\right)^{\frac{1}{p}}C_Q^{\frac{1}{p}}\left(2^{\frac{\left(Q-\lambda\right)}{p}}-1\right)^{-1}.\)
\end{lem}
\begin{proof}
    The case when \(k=0\) is trivial, for \(k>0\) from Lemma \ref{lem:average_rs} we get
    \begin{align*}
        \left|f_{B(x,r)}-f_{B(x,r/2^k)}\right|&\leq \sum_{m=0}^{k-1}\left|f_{B(x,r/2^m)}-f_{B(x,r/2^{m+1})}\right|\\
        &\leq 2^{1-\frac{1}{p}}r^{\frac{\lambda}{p}}\sum_{m=0}^{k-1}\frac{\left(1+\left(2\right)^{-\lambda}\right)^{\frac{1}{p}}}{\left(2^{\lambda m}\mu\left(B(x,r/2^{m+1})\right)\right)^{\frac{1}{p}}}\left[f\right]_{L^{p,\lambda}_C\left(X,\mu\right)}\\
        &\leq2^{1-\frac{1}{p}}r^{\frac{\lambda}{p}}\left(1+2^{-\lambda}\right)^{\frac{1}{p}}\sum_{m=0}^{k-1}\frac{1}{\left(2^{\lambda m}C_Q^{-1}\left(r/2^{m+1}\right)^Q\right)^{\frac{1}{p}}}\left[f\right]_{L^{p,\lambda}_C\left(X,\mu\right)}\\
        &=2^{1+\frac{Q-1}{p}}r^{\frac{\lambda-Q}{p}}\left(1+2^{-\lambda}\right)^{\frac{1}{p}}C_Q^{\frac{1}{p}}\frac{2^{\frac{\left(Q-\lambda\right)k}{p}}-1}{2^{\frac{\left(Q-\lambda\right)}{p}}-1}\left[f\right]_{L^{p,\lambda}_C\left(X,\mu\right)}.
    \end{align*}
\end{proof}
\section{Main results}

\begin{tw}\label{thm:embMorrCampLeb}
    Let \(\left(X,d,\mu\right)\) be a metric-measure space and \(p\in\left[1,\infty\right)\), then
\begin{enumerate}
\item[i)]  $L^{p,0}_M\left(X,\mu\right)=L^{p,0}_C\left(X,\mu\right)=L^p\left(X,\mu\right)$
    and the norms are equivalent; 
\item[ii)] If \(\left(X,d\right)\) is bounded and $\lambda\in\left[0,\infty\right)$, then 
    \[L^{p,\lambda}_M\left(X,\mu\right) \xhookrightarrow{} L^{p,\lambda}_C\left(X,\mu\right).\]
\end{enumerate}
\end{tw}

\begin{proof}
  $i)$ Firstly, let us observe that from the very definitions, we get that \(L^{p,0}_M\left(X,\mu\right)\subseteq L^p\left(X,\mu\right)\) and \(L^{p,0}_C\left(X,\mu\right)\subseteq L^p\left(X,\mu\right)\). We will show the opposite inclusions. Let \(f\in L^p\left(X,\mu\right)\). Then, 
\[
	\left\|f\right\|_{L^{p,0}_M\left(X,\mu\right)}=\sup_{x\in X, r>0}\left(\frac{1}{r^0}\int_{B(x,r)}\left|f\right|^p\dd\mu\right)^{\frac{1}{p}}=\left(\int_X\left|f\right|^p\dd\mu\right)^{\frac{1}{p}}=\left\|f\right\|_{L^p\left(X,\mu\right)}.
\]
    Moreover, by the Minkowski inequality and the Jensen inequality, we get
    \begin{align*}
        \left[f\right]_{L^{p,0}_C\left(X,\mu\right)}&=\sup_{x\in X, r>0}\left(\frac{1}{r^{0}}\int_{B(x,r)}\left|f-f_{B(x,r)}\right|^p\dd\mu\right)^{\frac{1}{p}}\\
               &\leq \left(\int_X\left|f\right|^p\dd\mu\right)^{\frac{1}{p}}+\sup_{x\in X, r>0}\left(\int_{B(x,r)}\left(\left|f\right|^p\right)_{B\left(x,r\right)}\dd\mu\right)^{\frac{1}{p}}\\
             &= 2\left(\int_X\left|f\right|^p\dd\mu\right)^{\frac{1}{p}}=2\left\|f\right\|_{L^p\left(X,\mu\right)}.
    \end{align*}
    Hence, we have 
    \[\left\|f\right\|_{L^p\left(X,\mu\right)}\leq \left\|f\right\|_{L^{p,0}_C\left(X,\mu\right)}=\left\|f\right\|_{L^p\left(X,\mu\right)}+\left[f\right]_{L^{p,0}_C\left(X,\mu\right)}\leq 3\left\|f\right\|_{L^p\left(X,\mu\right)}.\]

   $ii)$ Let \(f\in L^{p,\lambda}_M\left(X,\mu\right)\). From Lemma \ref{lem:equivOfNorms} we get
    \begin{align*}    
    \left\|f\right\|_{L^{p,\lambda}_C\left(X,\mu\right)}&=\left\|f\right\|_{L^p\left(X,\mu\right)}+\left[f\right]_{L^{p,\lambda}_C\left(X,\mu\right)}\leq 2\left(\left\|f\right\|_{L^p\left(X,\mu\right)}+\vertiii{f}_{p,\lambda}\right)\\
    &=2\left(\left\|f\right\|_{L^p\left(X,\mu\right)}+\left(\sup_{x\in X, r>0}\frac{1}{r^{\lambda}}\inf_{c\in\mathbb{R}}\int_{B(x,r)}\left|f-c\right|^p\dd\mu\right)^{\frac{1}{p}}\right)\leq 2\left(\left\|f\right\|_{L^p\left(X,\mu\right)}+\left\|f\right\|_{L^{p,\lambda}_M\left(X,\mu\right)}\right).
    \end{align*}
    To bound the norm \(\left\|f\right\|_{L^p\left(X,\mu\right)}\), let \(x_0\in X\), then we have
    \begin{align*}
        \left\|f\right\|_{L^p\left(X,\mu\right)}=\left(\int_X \left|f\right|^p\dd \mu\right)^{\frac{1}{p}}=\left(\int_{B(x_0,2\diam X)} \left|f\right|^p\dd \mu\right)^{\frac{1}{p}}\leq \left(2\diam X \right)^{\frac{\lambda}{p}}\left\|f\right\|_{L^{p,\lambda}_M\left(X,\mu\right)}.
    \end{align*}

\end{proof}
\begin{tw}\label{thm:embCampMorr}
    Let \(\left(X,d,\mu\right)\) be a bounded metric-measure space. 
    \begin{enumerate}
        \item \label{it:embCampMorr(i)}If the measure \(\mu\) is Ahlfors \(Q\)-regular, \(p\in\left[1,\infty\right)\) and \(\lambda\in \left[0,Q\right)\) then the following embedding holds
        \[
			L^{p,\lambda}_C\left(X,\mu\right)\xhookrightarrow{}L^{p,\lambda}_M\left(X,\mu\right).
	\]
        \item \label{it:embCampMorr(ii)}Let \(p\in\left[1,\infty\right)\) and \(\lambda \in (0, \infty)\). If the following embedding holds
        \[
		L^{p,\lambda}_C\left(X,\mu\right)\xhookrightarrow{}L^{p,\lambda}_M\left(X,\mu\right),
\]
        then the measure \(\mu\) is upper \(\lambda\)-regular.
    \end{enumerate}
\end{tw}
\begin{remark}
Let us stress that $(1)$ in Theorem \ref{thm:embCampMorr} does not hold for $\lambda =Q$. Indeed, if we take $((-1,1), |.|, \lambda)$, then the function $x \mapsto \log|x|$ belongs to $L^{1,1}_C\left((-1,1), |.|, \lambda \right)$, but in light of Theorem \ref{de} we have $L^{\infty}\left((-1,1), |.|, \lambda \right)= L_M^{1,1}\left((-1,1), |.|, \lambda\right)$.
\end{remark}
\begin{q}
Can we replace 'upper' by 'lower' in Theorem \ref{thm:embCampMorr} (2)?
\end{q}
\begin{q}
Let us suppose that $L^{p,\lambda}_C\left(X,\mu\right)\xhookrightarrow{}L^{p,\lambda}_M\left(X,\mu\right)$ for some $p \in [1,\infty)$ and every $\lambda \in [0,Q)$. Is it true that $\mu$ is upper $Q$-regular?
\end{q}
\begin{proof}
    \eqref{it:embCampMorr(i)} Let \(f\in L^{p,\lambda}_C\left(X,\mu\right)\),  \(x\in X, r>0\), and \(\lambda\in\left(0,Q\right)\). Let \(D:=\diam X\), we shall estimate the quantity 
\[
\left(\frac{1}{r^{\lambda}}\int_{B(x,r)}\left|f\right|^p\dd\mu \right)^{\frac{1}{p}}.
\]
We shall consider two cases.

If \(r\geq D\), then we have the following inequality
    \begin{align*} 
	\left(\frac{1}{r^{\lambda}}\int_{B(x,r)}\left|f\right|^p\dd\mu \right)^{\frac{1}{p}}\leq \frac{1}{r^{\lambda/p}}\left\|f\right\|_{L^p\left(X,\mu\right)} \leq D^{-\lambda/p}\left\|f\right\|_{L^p\left(X,\mu\right)}.
\end{align*}

If \(r<D\), then we fix \(k\in\mathbb{N}\) such that \(\frac{D}{2^{k+1}}\leq r<\frac{D}{2^k}\). We have 
 \begin{align}\label{ino}
        \int_{B(x,r)}\left|f\right|^p\dd\mu &\leq 2^{p-1}\left(\int_{B(x,r)}\left|f-f_{B(x,r)}\right|^p\dd\mu+\int_{B(x,r)}\left|f_{B(x,r)}\right|^p\dd\mu\right) \nonumber\\
        &\leq2^{p-1}\left(r^{\lambda}\left[f\right]_{L^{p,\lambda}_C\left(X,\mu\right)}^p + \int_{B(x,r)}\left|f_{B(x,r)}\right|^p\dd\mu\right). 
    \end{align}
On the other hand        
\begin{align}\label{ineq:proofCampMorey}
        \left|f_{B(x,r)}\right|\leq\left|f_X\right|+\left|f_X-f_{B\left(x,D/2^k\right)}\right|+\left|f_{B\left(x,D/2^k\right)}-f_{B(x,r)}\right|.
    \end{align}
Moreover, from Lemma \ref{lem:averages_k} we have
    \begin{align} \label{pierwsza}
        \left|f_X-f_{B\left(x,D/2^k\right)}\right|&=\left|f_{B(x,2D)}-f_{B\left(x,D/2^k\right)}\right|\leq C\left(2D\right)^{\frac{\lambda-Q}{p}}\left(2^{\frac{\left(Q-\lambda\right)\left(k+1\right)}{p}}-1\right)\left[f\right]_{L^{p,\lambda}_C\left(X,\mu\right)}
    \end{align}
    and Lemma \ref{lem:average_rs} yields 
    \begin{align}\label{druga}
        \left|f_{B\left(x,D/2^k\right)}-f_{B(x,r)}\right|\leq2^{1-\frac{1}{p}}\frac{\left(r^{\lambda}+\left(D2^{-k}\right)^{\lambda}\right)^{\frac{1}{p}}}{\left(\mu\left(B(x,r)\right)\right)^{\frac{1}{p}}}\left[f\right]_{L^{p,\lambda}_C\left(X,\mu\right)}\leq2^{1-\frac{1}{p}}\left(\frac{r^{\lambda}+\left(2r\right)^{\lambda}}{\mu\left(B(x,r)\right)}\right)^{\frac{1}{p}}\left[f\right]_{L^{p,\lambda}_C\left(X,\mu\right)}.
    \end{align}
   Therefore, integrating inequality (\ref{ineq:proofCampMorey}) over $B(x,r)$, gathering (\ref{pierwsza}) with (\ref{druga}) and using regularity of the measure, we get
    \begin{align*}
     &   \int_{B(x,r)}\left|f_{B(x,r)}\right|^p\dd\mu \leq3^{p-1}\left[\int_{B(x,r)}\left|f_X\right|^p+\left|f_X-f_{B\left(x,D/2^k\right)}\right|^p+\left|f_{B\left(x,D/2^k\right)}-f_{B(x,r)}\right|^p\dd\mu\right]\\
              &\leq 3^{p-1}\mu\left(B\left(x,r\right)\right)\Bigg[\fint_X \left|f\right|^p\dd \mu+\Bigg(C^p\left(2D\right)^{\lambda-Q}\left(2^{\frac{\left(Q-\lambda\right)\left(k+1\right)}{p}}-1\right)^p
        \left.\left.+2^{p-1}\frac{r^{\lambda}+\left(2r\right)^{\lambda}}{\mu\left(B(x,r)\right)}\right)\left[f\right]_{L^{p,\lambda}_C\left(X,\mu\right)}^p\right]\\
                &\leq 3^{p-1}\left[\frac{C_Qr^Q}{\mu\left(X\right)}\left\|f\right\|_{L^p\left(X,\mu\right)}^p+\left(C^p C_Qr^{\lambda}\left(1-2^{\frac{\left(\lambda-Q\right)\left(k+1\right)}{p}}\right)^p\right.\right.
        +2^{p-1}\left(r^{\lambda}+\left(2r\right)^{\lambda}\right)\bigg)\left[f\right]_{L^{p,\lambda}_C\left(X,\mu\right)}^p\bigg]\\
         &\leq 3^{p-1}\left[\frac{C_Qr^Q}{\mu\left(X\right)}\left\|f\right\|_{L^p\left(X,\mu\right)}^p+\left(C^p C_Qr^{\lambda}\left(1-\left(2D\right)^{\frac{\lambda-Q}{p}}r^{\frac{Q-\lambda}{p}}\right)^p\right.\right.
        +2^{p-1}\left(r^{\lambda}+\left(2r\right)^{\lambda}\right)\Bigg)\left[f\right]_{L^{p,\lambda}_C\left(X,\mu\right)}^p\Bigg].
    \end{align*}
Hence, gathering (\ref{ino}) with the above inequality, for $r<D$ we have
   \begin{align*}
        &\left(\frac{1}{r^{\lambda}}\int_{B(x,r)}\left|f\right|^p\dd\mu\right)^{\frac{1}{p}}\\
        \leq& \begin{multlined}[t]
            2^{1-\frac{1}{p}}\bigg(\left[f\right]_{L^{p,\lambda}_C\left(X,\mu\right)}^p\\+3^{p-1}\left[\frac{C_Qr^{Q-\lambda}}{\mu\left(X\right)}\left\|f\right\|_{L^p\left(X,\mu\right)}^p+\left(C^p C_Q\left(1-\left(2D\right)^{\frac{\lambda-Q}{p}}r^{\frac{Q-\lambda}{p}}\right)^p\right.+2^{p-1}\left(1+2^{\lambda}\right)\bigg)\left[f\right]_{L^{p,\lambda}_C\left(X,\mu\right)}^p\right]\bigg)^{\frac{1}{p}}
        \end{multlined}\\
        \leq&2^{1-\frac{1}{p}}\left(\left[f\right]_{L^{p,\lambda}_C\left(X,\mu\right)}^p+3^{p-1}\left[\frac{C_Q D^{Q-\lambda}}{\mu\left(X\right)}\left\|f\right\|_{L^p\left(X,\mu\right)}^p+\left(C^p C_Q+2^{p-1}\left(1+2^{\lambda}\right)\right)\left[f\right]_{L^{p,\lambda}_C\left(X,\mu\right)}^p\right]\right)^{\frac{1}{p}}\\
        \leq&2^{1-\frac{1}{p}}\left(3^{1-\frac{1}{p}}\left(\frac{C_Q D^{Q-\lambda}}{\mu\left(X\right)}\right)^{\frac{1}{p}}\left\|f\right\|_{L^p\left(X,\mu\right)}+\left(3^{p-1}\left(C^p C_Q+2^{p-1}\left(1+2^{\lambda}\right)\right)+1\right)^{\frac{1}{p}}\left[f\right]_{L^{p,\lambda}_C\left(X,\mu\right)}\right)
        \end{align*}
Finally, having in mind both cases we obtain
     \begin{align*}
        \left\|f\right\|_{L^{p,\lambda}_M\left(X,\mu\right)}=& \sup_{x\in X,\, r>0}\left(\frac{1}{r^{\lambda}}\int_{B(x,r)}\left|f\right|^p\dd\mu\right)^{\frac{1}{p}}\\
=&\max\left\{\sup_{x\in X,\, 0<r<D}\left(\frac{1}{r^{\lambda}}\int_{B(x,r)}\left|f\right|^p\dd\mu\right)^{\frac{1}{p}},\right.
        \left.\sup_{x\in X,\, r\geq D}\left(\frac{1}{r^{\lambda}}\int_{B(x,r)}\left|f\right|^p\dd\mu\right)^{\frac{1}{p}}\right\}\\
        \leq &\hat{C}\left\|f\right\|_{L^{p,\lambda}_C\left(X,\mu\right)},
    \end{align*}
    where we denoted
    \begin{align*}
        \hat{C}:= \max\left\{2^{1-\frac{1}{p}}3^{1-\frac{1}{p}}\left(\frac{C_Q D^{Q-\lambda}}{\mu\left(X\right)}\right)^{\frac{1}{p}},2^{1-\frac{1}{p}}\left(3^{p-1}\left(C^p C_Q+2^{p-1}\left(1+2^{\lambda}\right)\right)+1\right)^{\frac{1}{p}},D^{-\frac{\lambda}{p}}\right\}.
    \end{align*}
    \eqref{it:embCampMorr(ii)} Let \(C\) denote the norm of the embedding operator $L^{p,\lambda}_C\left(X,\mu\right)\xhookrightarrow{}L^{p,\lambda}_M\left(X,\mu\right)$. Then for \(x\in X, r>0\), we get
    \begin{align*}
        \frac{\mu\left(B\left(x,r\right)\right)}{r^{\lambda}}\leq\left\|1\right\|_{L^{p,\lambda}_M\left(X,\mu\right)}^p\leq C^p\left\|1\right\|_{L^{p,\lambda}_C\left(X,\mu\right)}^p=C^p\left(\left\|1\right\|_{L^p\left(X,\mu\right)}+\leftindex^C{\left[1\right]}_{p,\lambda}\right)^p=C^p\left\|1\right\|_{L^p\left(X,\mu\right)}^p=C^p\mu\left(X\right).
    \end{align*}
\end{proof}

\begin{tw} \label{de}
    Let \(\left(X,d,\mu\right)\) be a bounded metric-measure space and let \(p\in\left[1,\infty\right), Q\in\left(0,\infty\right)\). 
    \begin{enumerate}
        \item\label{it:embedLinfMorr} The measure \(\mu\) is upper \(Q\)-regular if and only if the following embedding holds
        \[L^{\infty}\left(X,\mu\right)\hookrightarrow L^{p,Q}_M\left(X,\mu\right).\]
        \item\label{it:equalLinfMorr} If the measure \(\mu\) is Ahlfors \(Q\)-regular and Borel regular, then
        \[L^{\infty}\left(X,\mu\right)= L_M^{p,Q}\left(X,\mu\right)\]
        and the norms are equivalent.
    \end{enumerate}
\end{tw}

\begin{proof}
    Let us begin with the 'if' part of the implication in \eqref{it:embedLinfMorr}. Let \(f\in L^{\infty}\left(X,\mu\right)\). From the boundness of the space \(X\) we get that \(f\in L^p\left(X,\mu\right)\). Moreover, since $\mu$ is upper $Q$-regular we have
    \begin{align*}
        \left\|f\right\|_{L^{p,Q}_M\left(X,\mu\right)}=\sup_{x\in X, r>0}\left(\frac{1}{r^{Q}}\int_{B\left(x,r\right)}\left|f\right|^p\dd\mu\right)^{\frac{1}{p}}\leq \sup_{x\in X, r>0}\left(\left\|f\right\|^p_{\infty}\frac{\mu\left(B\left(x,r\right)\right)}{r^Q}\right)^{\frac{1}{p}}\leq C_Q^{\frac{1}{p}}\left\|f\right\|_{L^{\infty}\left(X,\mu\right)}.
    \end{align*}
    For the proof of the 'only if' part, let us denote by $C$ the constant of the embedding $L^{\infty}\left(X,\mu\right)\hookrightarrow L^{p,Q}_M\left(X,\mu\right)$. Fix \(x\in X, r>0\). Taking \(f\equiv 1\), we get 
    \begin{align*}
        \frac{\mu\left(B\left(x,r\right)\right)}{r^Q}=\frac{1}{r^Q}\int_{B\left(x,r\right)}\left|f\right|^p\dd\mu\leq\left\|f\right\|_{L^{p,Q}_M\left(X,\mu\right)}^p\leq C^p\left\|f\right\|_{L^{\infty}\left(X,\mu\right)}=C^p.
    \end{align*}
    Moving to the proof of \eqref{it:equalLinfMorr}, let us observe that from \eqref{it:embedLinfMorr} we have that the embedding operator
    \[\mathcal{I}:L^{\infty}\left(X,\mu\right)\to L^{p,Q}_M\left(X,\mu\right)\]
    is bounded.  We will show that it is also surjective. By the way of contradiction, let us suppose that there exists a function \(f\in L^{p,Q}_M\left(X,\mu\right)\setminus L^{\infty}\left(X,\mu\right)\). For \(n\in\mathbb{N}\) let us denote \(A_n=\left\{x\in X: \left|f\left(x\right) \right|>n\right\}\). Since $f \notin L^{\infty}\left(X,\mu\right)$,  we have \(\mu\left(A_n\right)>0\) for \(n\in\mathbb{N}\). Having in mind that the measure is Borel regular and doubling, from the Lebesgue differentiation theorem \cite{H} we get that there exists a set of full measure \(F\subset X\) such that 
    \[ 
\lim_{r\to0} \fint_{B\left(x, r\right)} \left|f\right|^p \dd \mu=\left|f\left(x\right)\right|^p \,\,\, \text{for} \,\,\, x \in F.
 \]
    It is clear that for any \(n\in\mathbb{N}\) we have \(\mu\left(A_n\cap F\right)>0\). In particular $A_n\cap F \neq \emptyset$. Let us fix \(n\in \mathbb{N}\) and \(x_n\in A_n\cap F\), then we get
    \begin{align*}
        \left\|f\right\|_{L^{p,Q}_M\left(X,\mu\right)}&=\sup_{x\in X, r>0}\left(\frac{1}{r^Q} \int_{B\left(x,r\right)}\left|f\right|^p\dd\mu\right)^{\frac{1}{p}}\geq \limsup_{r\to 0}\left(\frac{1} {r^Q}\int_{B\left(x_n,r\right)}\left|f\right|^p\dd\mu\right)^{\frac{1}{p}}\\
        &\geq \limsup_{r\to 0}\left(\frac{C_Q^{-1}} {\mu\left(B\left(x_n,r\right)\right)}\int_{B\left(x_n,r\right)}\left|f\right|^p\dd\mu\right)^{\frac{1}{p}} \\
        & = \left(C_Q ^{-1}\left|f\left(x_n \right)\right|^p\right) ^{\frac{1}{p}} \geq C_Q^{-\frac{1}{p}}n.
    \end{align*}
    So we get that \( \left\|f\right\|_{L^{p,Q}_M\left(X,\mu\right)}=\infty\) and thus \(f\notin L^{p, Q}_M\left(X,\mu\right)\) is a contradiction. The Banach isomorphism theorem concludes the proof.
    \end{proof}
\begin{remark}
Let us observe that Theorem \ref{de} (2) is not optimal. If we assume that \(\left(X,d,\mu\right)\) is a bounded metric-measure space with Borel regular measure \(\mu\) which is upper \(Q\)-regular and almost everywhere lower $Q$-regular\footnote{The measure $\mu$ is almost everywhere lower $Q$-regular if there exists \(C>0\) such that for \(\mu\)-almost every \(x\in X\) there exists $\delta>0$ such that $Cr^Q\leq \mu\left(B\left(x,r\right)\right)$ for $r\in(0,\delta)$. }, then $L^{\infty}\left(X,\mu\right)= L_M^{p,Q}\left(X,\mu\right)$. Indeed, this observation follows from the inspection of the proof of Theorem \ref{de} and our assumptions on the measure allow us to conclude that  \(\left(X,d,\mu\right)\) is a Vitali metric-measure space, and in particular, the Lebesgue differentiation theorem holds on \(\left(X,d,\mu\right)\) (see \cite{HKST}).
\end{remark}

    \begin{tw} \label{camp}
        Let \(\left(X,d,\mu\right)\) be a metric-measure space, $Q \in (0,\infty)$, $\lambda \in (Q, \infty)$, $p \in [1, \infty)$, $a \in [0, \infty)$, $x_0 \in X$ and $\alpha=(\lambda-Q)/p$.
\begin{enumerate}
 \item[i)]  If $(X, d)$ is bounded and the measure \(\mu\) is upper \(Q\)-regular, then 
        \[C^{0,\alpha}\left(X\right)\xhookrightarrow{}L^{p,\lambda}_C\left(X,\mu+a\delta_{x_0} \right).\]
\item[ii)]  If $\mu$ is Borel regular, doubling, and lower \(Q\)-regular measure, then\footnote{This embedding means that there exists $C$ such that every $f \in L^{p,\lambda}_C\left(X,\mu\right)$ has a continuous representative $\tilde{f}$ and the inequality $\left\|\tilde{f}\right\|_{C^{0,\alpha}\left(X\right)} \leq C \left\|f\right\|_{L^{p,\lambda}_C\left(X,\mu\right)}$ holds.}
        \[L^{p,\lambda}_C\left(X,\mu + a \delta_{x_0}\right)\xhookrightarrow{}C^{0,\alpha}\left(X\right).\]
\end{enumerate}
            \end{tw}
\begin{q}
	 Let \(\left(X,d,\mu\right)\) be a metric-measure space such that the measure of every singleton is zero and let $Q \in (0,\infty)$, $\lambda \in (Q, \infty)$, $p \in [1, \infty)$.
	Suppose that $C^{0,(\lambda-Q)/p}\left(X\right)\xhookrightarrow{}L^{p,\lambda}_C\left(X,\mu \right)$. Is it true that $\mu$ is upper $Q-$regular?
\end{q}
    \begin{proof}
     $i)$   Let $\nu:=\mu+a\delta_{x_0}$, \(f\in C^{0,\alpha}\left(X\right)\) and let us fix \(x\in X, r>0\). We shall consider two cases. 

     Case I: \(x_0\notin B\left(x,r\right)\). Then, by the Jensen inequality, we have
         \begin{align*}
        \frac{1}{r^{\lambda}}\int_{B\left(x,r\right)}\left|f\left(y\right)-f_{B\left(x,r\right)}\right|^p\dd\nu\left(y\right) &=\frac{1}{r^{\lambda}}\int_{B\left(x,r\right)}\left|\fint_{B\left(x,r\right)}f\left(y\right)-f\left(z\right)\dd\nu\left(z\right)\right|^p\dd\nu\left(y\right)\\
&=\frac{1}{r^{\lambda}}\int_{B\left(x,r\right)}\left|\fint_{B\left(x,r\right)}f\left(y\right)-f\left(z\right)\dd\mu\left(z\right)\right|^p\dd\mu\left(y\right)\\
        &\leq\frac{1}{r^{\lambda}}\int_{B\left(x,r\right)}\fint_{B\left(x,r\right)}\left|f\left(y\right)-f\left(z\right)\right|^p\dd\mu\left(z\right)\dd\mu\left(y\right)\\
        &=\frac{1}{r^{\lambda}}\int_{B\left(x,r\right)}\fint_{B\left(x,r\right)}\left(\frac{\left|f\left(y\right)-f\left(z\right)\right|}{\left(d\left(y,z\right)\right)^{\alpha}}\right)^p\left(d\left(y,z\right)\right)^{\alpha p}\dd\mu\left(z\right)\dd\mu\left(y\right)\\
        &\leq\frac{1}{r^{\lambda}}\int_{B\left(x,r\right)}\fint_{B\left(x,r\right)}\left\|f\right\|_{C^{0,\alpha}\left(X\right)}^p \left(d\left(y,x\right)+d\left(x,z\right)\right)^{\alpha p}\dd\mu\left(z\right)\dd\mu\left(y\right)\\
        &\leq\frac{1}{r^{\lambda}}\int_{B\left(x,r\right)}\fint_{B\left(x,r\right)}\left\|f\right\|_{C^{0,\alpha}\left(X\right)}^p\left(2r\right)^{\alpha p}\dd\mu\left(z\right)\dd\mu\left(y\right)\\
        &=r^{-\lambda}\mu\left(B\left(x,r\right)\right)\left\|f\right\|_{C^{0,\alpha}\left(X\right)}^p\left(2r\right)^{\alpha p}\leq C_Q 2^{\alpha p}\left\|f\right\|_{C^{0,\alpha}\left(X\right)}^p r^{\alpha p+Q-\lambda}\\
        &=C_Q 2^{\alpha p}\left\|f\right\|_{C^{0,\alpha}\left(X\right)}^p.
    \end{align*}
   
        Case II: \(x_0\in B\left(x,r\right)\). Then we obtain
        \begin{align*}
          \frac{1}{r^{\lambda}}\int_{B\left(x,r\right)}\left|f-f_{B\left(x,r\right)}\right|^p\dd\nu
                    \leq& \frac{1}{r^{\lambda}}\int_{B\left(x,r\right)}\fint_{B\left(x,r\right)}\left|f\left(y\right)-f\left(z\right)\right|^p\dd\nu\left(z\right)\dd\nu\left(y\right)\\
          =&\frac{1}{r^{\lambda}}\fint_{B\left(x,r\right)}\left(\int_{B\left(x,r\right)}\left|f\left(y\right)-f\left(z\right)\right|^p\dd\mu\left(z\right)+a\left|f\left(y\right)-f\left(x_0\right)\right|^p\right)\dd\nu\left(y\right)\\
                    \leq& \frac{1}{r^{\lambda}}\fint_{B\left(x,r\right)}\left(\mu\left(B\left(x,r\right)\right)\left\|f\right\|_{C^{0,\alpha}\left(X\right)}^p\left(2r\right)^{\alpha p}+a\left|f\left(y\right)-f\left(x_0\right)\right|^p\right)\dd\nu\left(y\right)\\
          =& \frac{1}{r^{\lambda}}\left(\mu\left(B\left(x,r\right)\right)\left\|f\right\|_{C^{0,\alpha}\left(X\right)}^p\left(2r\right)^{\alpha p}+\fint_{B\left(x,r\right)}a\left|f\left(y\right)-f\left(x_0\right)\right|^p\dd\nu\left(y\right)\right)\\
          \leq& \begin{multlined}[t]
              \frac{1}{r^{\lambda}}C_Qr^Q\left\|f\right\|_{C^{0,\alpha}\left(X\right)}^p\left(2r\right)^{\alpha p}\\
              +\frac{1}{r^{\lambda}\nu\left(B\left(x,r\right)\right)}\left(\int_{B\left(x,r\right)}a\left|f\left(y\right)-f\left(x_0\right)\right|^p\dd\mu\left(y\right)\right)
          \end{multlined}\\
          \leq& \begin{multlined}[t]
              C_Q\left\|f\right\|_{C^{0,\alpha}\left(X\right)}^p2^{\alpha p}\\
              +\frac{1}{r^{\lambda}\nu\left(B\left(x,r\right)\right)}\left(\int_{B\left(x,r\right)}a\left\|f\right\|_{C^{0,\alpha}\left(X\right)}^p\left(d\left(y,x_0\right)\right)^{\alpha p}\dd\mu\left(y\right)\right)
          \end{multlined}\\
          \leq& C_Q\left\|f\right\|_{C^{0,\alpha}\left(X\right)}^p2^{\alpha p}+\frac{1}{r^{\lambda}\nu\left(B\left(x,r\right)\right)}\mu\left(B\left(x,r\right)\right)a\left\|f\right\|_{C^{0,\alpha}\left(X\right)}^p\left(2r\right)^{\alpha p}\\
          \leq& C_Q\left\|f\right\|_{C^{0,\alpha}\left(X\right)}^p2^{\alpha p}+\frac{a}{r^{\lambda}\left(\mu\left(B\left(x,r\right)\right)+a\right)}C_Qr^Q\left\|f\right\|_{C^{0,\alpha}\left(X\right)}^pr^{\alpha p}2^{\alpha p}\\
          \leq& C_Q2^{\alpha p+1}\left\|f\right\|_{C^{0,\alpha}\left(X\right)}^p.
     \end{align*}

    Gathering both cases we obtain
    \[\left[f\right]_{L^{p,\lambda}_C\left(X,\nu\right)}=\sup_{x\in X, r>0}\left(\frac{1}{r^{\lambda}}\int_{B(x,r)}\left|f-f_{B(x,r)}\right|^p\dd\nu\right)^{\frac{1}{p}}\leq C_Q^{\frac{1}{p}} 2^{\alpha+\frac{1}{p}}\left\|f\right\|_{C^{0,\alpha}\left(X\right)}.\]
    Moreover, we have
        \[\left\|f\right\|_{L^p\left(X,\nu\right)}=\left(\int_X\left|f\right|^p\dd\nu\right)^{\frac{1}{p}}\leq\left(\nu\left(X\right)\right)^{\frac{1}{p}}\left\|f\right\|_{C^{0,\alpha}\left(X\right)},\]
and the proof of $i)$ follows.

        $ii)$ Since $L^{p,\lambda}_C\left(X,\mu + a \delta_{x_0}\right)\xhookrightarrow{} L^{p,\lambda}_C\left(X,\mu \right)$, it is enough to prove $L^{p,\lambda}_C\left(X,\mu \right)\xhookrightarrow{}C^{0,\alpha}\left(X\right)$. First of all, we shall prove the following lemma.
\begin{lem}\label{lem:uniformLeb}
         Let \(\left(X,d,\mu\right)\) be a metric-measure space with Borel regular, doubling and lower \(Q\)-regular measure \(\mu\). If \(p\in\left[1,\infty\right), \lambda\in\left(Q,\infty\right)\), and \(f\in L_C^{p,\lambda}\left(X,\mu\right)\), then there exists a Borel function \(\tilde{f}\) such that \(\tilde{f}=f\) \(\mu\)-almost everywhere, \(\tilde{f}\left(x\right)=\lim_{r\to0^+}f_{B\left(x,r\right)}\) uniformly on \(X\) and the following inequality holds
        \[
\,\left|f_{B\left(x,r\right)}-F\left(x\right)\right|\leq K\left[f\right]_{L^{p,\lambda}_C\left(X,\mu\right)}r^{\frac{\lambda-Q}{p}}\,\,\, \text{for}\,\,\, x\in X, r\in\left(0,3\diam X\right),
\]
        where \(K=2^{1+\frac{Q-1}{p}}\left(1+2^{-\lambda}\right)^{\frac{1}{p}}C_Q^{\frac{1}{p}}\left(1-2^{\frac{\left(Q-\lambda\right)}{p}}\right)^{-1}.\)
    \end{lem}
    \begin{proof}
        Let \(f\in L^{p,\lambda}_C\left(X,\mu\right), x\in X, r\in\left(0, 3\diam X\right)\) and \(n, k\in\mathbb{N}\), then from Lemma \ref{lem:averages_k} we get 
        \begin{align} \label{K}
        \left|f_{B\left(x,r/2^n\right)}-f_{B\left(x,r/2^{n+k}\right)}\right|&\leq K\left(r/2^n\right)^{\frac{\lambda-Q}{p}}\left(1-2^{\frac{\left(Q-\lambda\right)k}{p}}\right)\left[f\right]_{L^{p,\lambda}_C\left(X,\mu\right)} \nonumber\\
        &\leq K\left(r/2^n\right)^{\frac{\lambda-Q}{p}}\left[f\right]_{L^{p,\lambda}_C\left(X,\mu\right)},
    \end{align}
    where \(K=2^{1+\frac{Q-1}{p}}\left(1+2^{-\lambda}\right)^{\frac{1}{p}}C_Q^{\frac{1}{p}}\left(1-2^{\frac{\left(Q-\lambda\right)}{p}}\right)^{-1}>0\). So \(\left(f_{B\left(x,r/2^n\right)}\right)_{n=0}^{\infty}\) is a Cauchy sequence. Let us denote \(\tilde{f}_r\left(x\right)=\lim_{n\to\infty}f_{B\left(x,r/2^n\right)}\). The function \(\tilde{f}_r\) is Borel as a limit of Borel maps. Fix \(0<r<s<3\diam X\) and \(k\in\mathbb{N}\), from Lemma \ref{lem:average_rs} we get
    \begin{align*}
      \left|f_{B\left(x,r/2^k\right)}-f_{B\left(x,s/2^k\right)}\right|&\leq 2^{1-\frac{1}{p}}\frac{\left(\left(r/2^k\right)^{\lambda}+\left(s/2^k\right)^{\lambda}\right)^{\frac{1}{p}}}{\left(\mu\left(B\left(x,r/2^k\right)\right)\right)^{\frac{1}{p}}}\left[f\right]_{L^{p,\lambda}_C\left(X,\mu\right)}\\
      &\leq 2^{1-\frac{1}{p}}C_Q^{\frac{1}{p}}\frac{\left(\left(r/2^k\right)^{\lambda}+\left(s/2^k\right)^{\lambda}\right)^{\frac{1}{p}}}{\left(r/2^k\right)^{\frac{Q}{p}}}\left[f\right]_{L^{p,\lambda}_C\left(X,\mu\right)}\\
      &=2^{1-\frac{1}{p}}C_Q^{\frac{1}{p}}\left(\frac{r^{\lambda}+s^{\lambda}}{r^Q}\right)^{\frac{1}{p}}2^{k\frac{Q-\lambda}{p}}\left[f\right]_{L^{p,\lambda}_C\left(X,\mu\right)}.
    \end{align*}
    By passing to the limit, we see that the function \(\tilde{f}_r\) does not depend on \(r\), so from this point we will just call it \(\tilde{f}\). From the Lebesgue differentiation theorem, we get that \(\tilde{f}=f\) \(\mu\)-almost everywhere on \(X\). Moreover, from (\ref{K}) with $n=0$ for any \(r\in\left(0,3\diam X\right)\) we get 
    \[\left|f_{B\left(x,r\right)}-f_{B\left(x,r/2^{k}\right)}\right| \leq Kr^{\frac{\lambda-Q}{p}}\left[f\right]_{L^{p,\lambda}_C\left(X,\mu\right)}.\]
   Hence, by passing to the limit, we get
        \[\left|f_{B\left(x,r\right)}-\tilde{f}\left(x\right)\right| \leq Kr^{\frac{\lambda-Q}{p}}\left[f\right]_{L^{p,\lambda}_C\left(X,\mu\right)}.\]
In particular, we get that the convergence is uniform.
    \end{proof}
        Let \(f\in L^{p,\lambda}_C\left(X,\mu\right) \) and denote the function from Lemma \ref{lem:uniformLeb} by \(\tilde{f}\). Furthermore, let us fix \(x,y\in X\) and denote \(r=2d\left(x,y\right)\).   We have 
        \begin{align}\label{eq:triangIneqHolder}
        \left|\tilde{f}\left(x\right)-\tilde{f}\left(y\right)\right|\leq\left|\tilde{f}\left(x\right)-f_{B\left(x,r\right)}\right|+\left|f_{B\left(x,r\right)}-f_{B\left(y,r\right)}\right|+\left|f_{B\left(y,r\right)}-\tilde{f}\left(y\right)\right|.
    \end{align}
    From Lemma \ref{lem:uniformLeb} we get 
\begin{eqnarray}\label{k1}
\left|\tilde{f}\left(x\right)-f_{B\left(x,r\right)}\right|\leq K\left[f\right]_{L^{p,\lambda}_C\left(X,\mu\right)}\left(2d\left(x,y\right)\right)^{\alpha}.
\end{eqnarray}
     Furthermore, we see that
      \begin{align*}
        \left|f_{B\left(x,r\right)}-f_{B\left(y,r\right)}\right|\leq \left|f_{B\left(x,r\right)}-f\right|+\left|f-f_{B\left(y,r\right)}\right|.
    \end{align*}
    By integrating both sides of the above inequality and using the definition of the Campanato seminorm, we get
        \begin{align*}
        \int_{B\left(x,r\right)\cap B\left(y,r\right)}\left|f_{B\left(x,r\right)}-f_{B\left(y,r\right)}\right|\dd\mu &\leq \int_{B\left(x,r\right)\cap B\left(y,r\right)}\left|f_{B\left(x,r\right)}-f\right|\dd\mu+\int_{B\left(x,r\right)\cap B\left(y,r\right)}\left|f-f_{B\left(y,r\right)}\right|\dd\mu\\
        &\leq \int_{B\left(x,r\right)}\left|f_{B\left(x,r\right)}-f\right|\dd\mu+\int_{B\left(y,r\right)}\left|f-f_{B\left(y,r\right)}\right|\dd\mu\\
        &\leq \left(\left(\mu\left(B\left(x,r\right)\right)\right)^{1-\frac{1}{p}}+\left(\mu\left(B\left(y,r\right)\right)\right)^{1-\frac{1}{p}}\right)r^{\frac{\lambda}{p}}\left[f\right]_{L^{p,\lambda}_C\left(X,\mu\right)}.
    \end{align*}
   Since the following inclusion holds
    \[B\left(x,r/2\right)\cup B\left(y,r/2\right)\subseteq B\left(x,r\right)\cap B\left(y,r\right),\]
     we get
    \begin{align*}
        \left|f_{B\left(x,r\right)}-f_{B\left(y,r\right)}\right|&\leq \left(\frac{\left(\mu\left(B\left(x,r\right)\right)\right)^{1-\frac{1}{p}}}{\mu\left(B\left(x,r\right)\cap B\left(y,r\right)\right)}+\frac{\left(\mu\left(B\left(y,r\right)\right)\right)^{1-\frac{1}{p}}}{\mu\left(B\left(x,r\right)\cap B\left(y,r\right)\right)}\right)r^{\frac{\lambda}{p}}\left[f\right]_{L^{p,\lambda}_C\left(X,\mu\right)}\\
        &\leq \left(\frac{\left(\mu\left(B\left(x,r\right)\right)\right)^{1-\frac{1}{p}}}{\mu\left(B\left(x,r/2\right)\right)}+\frac{\left(\mu\left(B\left(y,r\right)\right)\right)^{1-\frac{1}{p}}}{\mu\left(B\left(y,r/2\right)\right)}\right)r^{\frac{\lambda}{p}}\left[f\right]_{L^{p,\lambda}_C\left(X,\mu\right)}\\
        &\leq \left(\frac{C_{\mu}}{\left(\mu\left(B\left(x,r\right)\right)\right)^{\frac{1}{p}}}+\frac{C_{\mu}}{\left(\mu\left(B\left(y,r\right)\right)\right)^{\frac{1}{p}}}\right)r^{\frac{\lambda}{p}}\left[f\right]_{L^{p,\lambda}_C\left(X,\mu\right)}\\
        &\leq 2\frac{C_{\mu}}{\left(C_Q^{-1}r^Q\right)^{\frac{1}{p}}}r^{\frac{\lambda}{p}}\left[f\right]_{L^{p,\lambda}_C\left(X,\mu\right)}=2C_{\mu}C_Q^{\frac{1}{p}}\left(2d\left(x,y\right)\right)^{\alpha}\left[f\right]_{L^{p,\lambda}_C\left(X,\mu\right)},
    \end{align*}
where we applied the doubling condition and the lower $Q$-regularity of the measure $\mu$. 
    Next, plugging the above estimate with (\ref{k1}) to inequality \eqref{eq:triangIneqHolder} we get
    \begin{align*}
        \left|\tilde{f}\left(x\right)-\tilde{f}\left(y\right)\right|
        &\leq K\left(2d\left(x,y\right)\right)^{\alpha}\left[f\right]_{L^{p,\lambda}_C\left(X,\mu\right)} + 2C_{\mu}C_Q^{\frac{1}{p}}\left(2d\left(x,y\right)\right)^{\alpha}\left[f\right]_{L^{p,\lambda}_C\left(X,\mu\right)} + K\left(2d\left(x,y\right)\right)^{\alpha}\left[f\right]_{L^{p,\lambda}_C\left(X,\mu\right)}\\
        &=2^{1+\alpha}\left(K+C_{\mu}C_Q^{\frac{1}{p}}\right)\left(d\left(x,y\right)\right)^{\alpha}\left[f\right]_{L^{p,\lambda}_C\left(X,\mu\right)}.
    \end{align*}
    Hence, we have
	\begin{eqnarray}\label{seminot}    
\sup_{x,y\in X, x\neq y}\frac{\left|\tilde{f}\left(x\right)-\tilde{f}\left(y\right)\right|}{\left(d\left(x,y\right)\right)^{\alpha}}\leq 2^{1+\alpha}\left(K+C_{\mu}C_Q^{\frac{1}{p}}\right)\left[f\right]_{L^{p,\lambda}_C\left(X,\mu\right)}.
	\end{eqnarray}
    From the above inequality we have that \(\tilde{f}\) is continuous. We are left to bound \(\sup_{x\in X}\left|\tilde{f}\left(x\right)\right|\). Fix \(x\in X, r\in\left(0,\diam X\right)\). From Lemma \ref{lem:uniformLeb} we get
    \begin{align*}
        \left|\tilde{f}\left(x\right)\right|&\leq \left|F\left(x\right)-f_{B\left(x,r\right)}\right|+\left|f_{B\left(x,r\right)}\right|\leq  Kr^{\frac{\lambda-Q}{p}}\left[f\right]_{L^{p,\lambda}_C\left(X,\mu\right)}+\left(\fint_{B\left(x,r\right)}\left|f\right|^p\dd\mu\right)^{\frac{1}{p}}\\
&\leq Kr^{\frac{\lambda-Q}{p}}\left[f\right]_{L^{p,\lambda}_C\left(X,\mu\right)}+\frac{1}{\left(\mu\left(B\left(x,r\right)\right)\right)^{\frac{1}{p}}}\left\|f\right\|_{L^p\left(X,\mu\right)}\\
        &\leq Kr^{\frac{\lambda-Q}{p}}\left[f\right]_{L^{p,\lambda}_C\left(X,\mu\right)}+\frac{1}{\left(C_Q ^{-1}r^Q\right)^{\frac{1}{p}}}\left\|f\right\|_{L^p\left(X,\mu\right)}\leq\left(Kr^{\frac{\lambda-Q}{p}}+C_Q^{\frac{1}{p}}r^{-\frac{Q}{p}}\right)\left\|f\right\|_{L^{p,\lambda}_C\left(X,\mu\right)}.
    \end{align*}
    So we have
    \begin{align*}
        \left\|\tilde{f}\right\|_{C^{0,\alpha}\left(X\right)}\leq \left(\inf_{r\in\left(0, \diam X\right)}\left(Kr^{\frac{\lambda-Q}{p}}+C_Q^{\frac{1}{p}}r^{-\frac{Q}{p}}\right)+2^{1+\alpha}\left(K+C_{\mu}C_Q^{\frac{1}{p}}\right)\right)\left\|f\right\|_{L^{p,\lambda}_C\left(X,\mu\right)}.
    \end{align*}
    \end{proof}

\begin{ex}
     Let \(\left(\Omega, |\cdot|, \nu\right)\), where \(\Omega  =\left\{\left(x_1,x_2\right)\in\mathbb{R}^2:x_2\in\left[0,1\right], |x_1|\leq x_2^2\right\}, \nu= \lambda_{ \mres \Omega}\), and \(\left|\cdot\right|\) is the Euclidean distance. Then it is obvious that \(\nu\) is upper 2-regular and not lower 2-regular. Moreover, for \(p\in\left[1,+\infty\right), \lambda\in\left[2+p,+\infty\right)\) we have 
     \[L_C^{p,\lambda}\left(\Omega,\nu\right)\hookrightarrow{}C^{0,\alpha}\left(\Omega\right),\]
     where \(\alpha = \left(\lambda-2\right)/p\).
\end{ex}
\begin{proof}
Frist of all, let us observe that from Theorem \ref{camp} for every ball $B \subset \Omega$ we have that \(L_C^{p, \lambda}\left(B\right)\hookrightarrow C^{0,\alpha}\left(B\right)\). Moreover, for every \(u\in L_C^{p,\lambda}\left(\Omega,\nu\right)\) such that \(\left\|u\right\|_{L_C^{p,\lambda}\left(\Omega,\nu\right)}\leq 1\) we have
\begin{eqnarray}\label{seminorma1}
	 \left[u\right]_{C^{0,\alpha}\left(B\right)} \leq C,
\end{eqnarray}
where by (\ref{seminot}) the constant $C$ does not depend on $B$.

  Let us fix  \(f\in L_C^{p,\lambda}\left(\Omega,\nu\right)\) such that \(\left\|f\right\|_{L_C^{p,\lambda}\left(\Omega,\nu\right)}\leq 1\)  and let \(x=(x_1,x_2),y=(y_1,y_2)\) be two points from the interior of \(\Omega\). Without loss of generality, we can assume that \(|x_1|\leq|y_1|\). Define \(z=(x_1,y_2)\), then the line segments connecting \(x\) with \(z\) and \(z\) with \(y\) lie within the interior of \(\Omega\). Moreover, we have
    \[\left|x-z\right|+\left|z-y\right|\leq \sqrt{2} \left|x-y\right|.\]
    From the compactness of the line segment connecting \(x\) with \(z\) we can cover it with open balls \(\left(B_i\right)_{i=1}^n\) such that \(B_i\subset \Omega\) for each \(i=0,1,\dots,n\) and \(B_i\cap B_j\neq \emptyset\) if and only if \(\left|i-j\right|\leq 1\). Let us define a sequence of points \(\left(\xi_i\right)_{i=0}^n\) from the line segment connecting \(x\) with \(z\) such that \(\xi_0=x, \xi_n=z\) and \(\xi_i\in B_i\cap B_{i+1}\) for \(i=1,2,\dots,n-1\). Therefore, by (\ref{seminorma1}) we have
    \[\left|f\left(\xi_i\right)-f\left(\xi_{i+1}\right)\right|\leq C \left|\xi_i-\xi_{i+1}\right|^{\alpha}.\]
    Since  \(\alpha \geq1\), by superadditivity of the function \([0,\infty) \ni t\mapsto t^{\alpha}\) we have
    \[\left|f\left(x\right)-f\left(z\right)\right|\leq \sum_{i=0}^{n-1}\left|f\left(\xi_i\right)-f\left(\xi_{i+1}\right)\right| \leq C \sum_{i=0}^{n-1}\left|\xi_i-\xi_{i+1}\right|^{\alpha}\leq C \left(\sum_{i=0}^{n-1}|\xi_i-\xi_{i+1}|\right)^{\alpha}= C\left|x-z\right|^{\alpha}.\]
    Similarly
    \[\left|f\left(z\right)-f\left(y\right)\right|\leq C\left|z-y\right|^{\alpha}.\]
    Hence, we get
    \[\left|f\left(x\right)-f\left(y\right)\right|\leq C\left|x-z\right|^{\alpha}+C\left|z-y\right|^{\alpha}\leq 2^{\alpha/2}C \left|x-y\right|^{\alpha}.\]
    In the case when at least one of the points \(x,y\) lies on the boundary, the standard limiting argument concludes that
\begin{eqnarray}\label{seminorma}
	 \left[f\right]_{C^{0,\alpha}\left(X\right)} \leq 2^{\alpha/2}C.
\end{eqnarray}
Next, in order to control the supremum norm we fix $\tilde{x} \in \hbox{int} \Omega$ and $\tilde{r} \in (0,\infty)$ such $B(\tilde{x},\tilde{r})\subset  \Omega$. Then, by Theorem \ref{camp} there exists $C_{\tilde{x},\tilde{r}}$ which does not depend on $f$ and $|f(\tilde{x})|\leq C_{\tilde{x},\tilde{r}}$. Therefore, from (\ref{seminorma}), for every $x\in\Omega$ we have
\[
	|f(x)|\leq |f(x)-f(\tilde{x})|+|f(\tilde{x})| \leq 2^{\alpha/2}C (\hbox{diam} \Omega)^{\alpha}+C_{\tilde{x},\tilde{r}}.
\]  
Finally, we get
\[
	\|f\|_{C^{0,\alpha}(\Omega)}\leq 2^{\alpha/2}C +2^{\alpha/2}C (\hbox{diam} \Omega)^{\alpha}+C_{\tilde{x},\tilde{r}}.
\]
\end{proof}

\begin{tw} \label{dolne}
        Let \(\left(X,d,\mu\right)\) be a compact metric-measure space such that \(\left(X,d\right)\) is \(\eta\)-uniformly perfect metric space, $Q \in (0, \infty)$ and
one of the following conditions holds
\begin{itemize}
\item[{I)}]   \(p\in\left[1,\infty\right), \lambda\in\left(Q,p\right]\);
\item[{II)}] \(p\in\left[1,\infty\right), \lambda\in\left(Q,Q+p\right)\) and \(\mu\) is upper \(Q\)-regular.
\end{itemize}
If the measure \(\mu\) is not lower \(Q\)-regular, then there exists a function
        \[f\in L^{p,\lambda}_C\left(X,\mu\right)\setminus C^{0, \alpha}\left(X\right),\]
        where\footnote{It means that $f$ has no continuous representative belonging to $C^{0, \alpha}\left(X\right)$.} \(\alpha = (\lambda-Q)/p.\)
    \end{tw}

\begin{q}
Is it true that Theorem \ref{dolne} holds for $\lambda\in\left(Q,Q+p\right)$ without the assumption that the measure is upper \(Q\)-regular?
\end{q}

    \begin{proof}    The methods of the proof are based on the paper \cite{O}.
        Let \(\left(x_n\right)_{n=1}^{\infty}\subseteq X, \left(r_n\right)_{n=1}^{\infty}\subseteq \left(0,3\diam X\right)\) be sequences from Lemma \ref{lem:auxLem}. Let us define the following sequence of functions
        \begin{align*}
            u_n\left(x\right)=\begin{cases}
            r_n/10-d\left(x,x_n\right) & \text{if } x\in B\left(x_n,r_n/10\right);\\
            0 & \text{if } x\notin B\left(x_n,r_n/10\right);
            \end{cases}
        \end{align*}
        and the following function
        \[f\left(x\right)=\sum_{n=1}^{\infty}\frac{10^{\kappa \alpha n}}{r_n^{1-\alpha}}u_n\left(x\right),\]
        where \(\kappa=\min\left\{\left(2\alpha p\right)^{-1},\frac{1-\alpha}{\alpha\left(p+Q\right)}\right\}>0\). Furthermore, let $A:=\bigcup_{n=1}^{\infty}B(x_n, r_n)$. The function \(f\) is well-defined, as the balls \(B\left(x_n,r_n\right)\) are pairwise disjoint, and $f|_A$ is continuous on $A$.

First of all, we shall prove that $f$ has no continuous representative belonging to $C^{0, \alpha}\left(X\right)$. 
        From uniform perfectness we obtain that for any \(n\in\mathbb{N}\) there exists \(z_n\in X\) such that
        \[z_n\in B\left(x_n,r_n/10\right)\setminus B\left(x_n,\eta r_n/10\right).\]
        Then we get
        \begin{align*}
            \sup_{x,y\in X, x\neq y}\frac{\left|f\left(x\right)-f\left(y\right)\right|}{d\left(x,y\right)^{\alpha}}&\geq \sup_{x,y\in A, x\neq y}\frac{\left|f\left(x\right)-f\left(y\right)\right|}{d\left(x,y\right)^{\alpha}}\geq \limsup_{n\to\infty}\frac{\left|f\left(x_n\right)-f\left(z_n\right)\right|}{d\left(x_n,z_n\right)^{\alpha}}\\&= \limsup_{n\to\infty}\frac{10^{\kappa \alpha n}d\left(x_n,z_n\right)^{1-\alpha}}{r_n^{1-\alpha}}
            \geq\limsup_{n\to\infty}10^{\kappa\alpha n-1}\eta^{1-\alpha}=\infty.  
        \end{align*}
        Therefore \(f\notin C^{0,\alpha}\left(X\right)\) and \(f|_A \notin C^{0,\alpha}\left(A\right)\). Since the measure of every nonempty open set is positive, every continuous representative of $f$ coincide with $f$ on $A$. Therefore, there are no representatives of $f$ belonging to $ C^{0,\alpha}\left(X\right)$.

Next, we will show that \(f\in L^{p,\lambda}_C\left(X,\mu\right)\). Firstly, we will prove that
        \begin{equation}\label{eq:BnEstimate}
        \int_{B\left(x_n,r_n/10\right)}\left|f\right|^p\dd\mu\leq Cr_n^{\lambda}10^{\left(\kappa\alpha p-1\right)n},
    \end{equation}
    where \(C=20^Q\eta^{-Q}\) is the constant from Lemma \ref{lem:auxLem}. Indeed, by Lemma \ref{lem:auxLem} we have
    \begin{align*}
        \int_{B\left(x_n,r_n/10\right)} \left|f\left(x\right)\right|^p\dd\mu\left(x\right)
        &=\int_{B\left(x_n,r_n/10\right)} \left|\frac{10^{\kappa \alpha n}}{r_n^{1-\alpha}}\left(r_n/10-d\left(x,x_n\right)\right)\right|^p\dd\mu\left(x\right)\\
        &\leq \mu\left(B\left(x_n,r_n/10\right)\right)10^{\kappa \alpha np}r_n^{\alpha p}10^{-p}\\
        &\leq  C10^{-n}r_n^Q 10^{\kappa \alpha np}r_n^{\alpha p}= Cr_n^{\lambda}10^{\left(\kappa\alpha p-1\right)n}.
    \end{align*}
    From inequality \eqref{eq:BnEstimate} we get 
       \begin{align*}
        \int_X \left|f\right|^p\dd\mu &= \int_{\bigcup_{n=1}^{\infty}B\left(x_n,r_n/10\right)} \left|f\right|^p\dd\mu\\
        &\leq \sum_{n=1}^{\infty}Cr_n^{\lambda}10^{\left(\kappa\alpha p-1\right)n}\leq Cr_1^{\lambda}\sum_{n=1}^{\infty}10^{\left(\kappa\alpha p-1\right)n}<\infty,
    \end{align*}
    because \(\kappa\leq \left(2\alpha p\right)^{-1}\). Therefore \(f\in L^p\left(X,\mu\right)\).
    
Next, we shall show that the Campanato seminorm of $f$ is finite. For this purpose, we fix \(x\in X, r>0\). Observe that if there exists a ball \(B\left(x_i,r_i/10\right)\) such that
    \begin{equation}\label{eq:BiCondition}
        B\left(x,r\right)\cap B\left(x_i,r_i/10\right)\neq \emptyset \text{ and } r_i/10>r,
    \end{equation}
    then
    \[\forall_{k\in\mathbb{Z}_+, k\neq i}B\left(x,r\right)\cap B\left(x_k,r_k/10\right)=\emptyset.\]
    Indeed, let us suppose that there exists a ball such that it satisfies the condition \eqref{eq:BiCondition} and that \(i\) is the smallest index for which it holds. If $k<i$, then by monotonicity of \(\left(r_n\right)_{n=1}^{\infty}\) and minimality of $i$, we get $B\left(x,r\right)\cap B\left(x_k,r_k/10\right)=\emptyset.$ On the other hand we have 
$B\left(x,r\right)\subseteq B\left(x_i, 3r_i/10\right)$. Thus, if $k>i$, by Lemma \ref{lem:auxLem} we get 
    \[B\left(x,r\right)\cap B\left(x_k,r_k/10\right)\subseteq B\left(x_i,3r_i/10\right)\cap B\left(x_k,r_k/10\right)\subseteq B\left(x_i,3r_i\right)\cap B\left(x_k,r_k\right)=\emptyset.\]

   In order to control the Campanato seminorm, we consider two cases:
    \begin{equation}\label{eq:BiNorm}
        \inf_{c\in\mathbb{R}}\int_{B\left(x,r\right)}\left|f-c\right|^p\dd\mu=\inf_{c\in\mathbb{R}} \left(\int_{B\left(x,r\right)\cap B\left(x_i,r_i/10\right)}\left|f-c\right|^p\dd\mu+\int_{B\left(x,r\right)\setminus B\left(x_i,r_i/10\right)}\left|c\right|^p\dd\mu\right),
    \end{equation}
    where \(B\left(x_i,r_i/10\right)\) satisfies condition \eqref{eq:BiCondition} or
    \begin{align}\label{eq:notBiNorm}
        \begin{split}
            \inf_{c\in\mathbb{R}}\int_{B\left(x,r\right)}\left|f-c\right|^p\dd\mu&\leq\int_{B\left(x,r\right)}\left|f\right|^p\dd\mu=\int_{\bigcup_{\left\{j\in\mathbb{Z}_+:r_j/10\leq r\right\}}B\left(x_j,r_j/10\right)\cap B\left(x,r\right)}\left|f\right|^p\dd\mu\\
            &\leq\sum_{\left\{j\in\mathbb{Z}_+:r_j/10\leq r\right\}}\int_{B\left(x_j,r_j/10\right)}\left|f\right|^p\dd\mu.
        \end{split}
    \end{align}
    Let us begin with the case from equality \eqref{eq:BiNorm}. From inequality \eqref{eq:BnEstimate} we get 
    \begin{equation}\label{eq:BxrEstimate1}
        \inf_{c\in\mathbb{R}} \left(\int_{B\left(x,r\right)\cap B\left(x_i,r_i/10\right)}\left|f-c\right|^p\dd\mu+\int_{B\left(x,r\right)\setminus B\left(x_i,r_i/10\right)}\left|c\right|^p\dd\mu\right)\leq\int_{B\left(x_i,r_i/10\right)}\left|f\right|^p\dd\mu\leq Cr_i^{\lambda}10^{\left(\kappa\alpha p-1\right)i}.
    \end{equation}
    Now we will estimate it differently, observe that
    \begin{multline*}
                \inf_{c\in\mathbb{R}} \left(\int_{B\left(x,r\right)\cap B\left(x_i,r_i/10\right)}\left|f-c\right|^p\dd\mu+\int_{B\left(x,r\right)\setminus B\left(x_i,r_i/10\right)}\left|c\right|^p\dd\mu\right)\leq\\
        \int_{B\left(x,r\right)\cap B\left(x_i,r_i/10\right)}\left|f-\tilde{c}\right|^p\dd\mu+\int_{B\left(x,r\right)\setminus B\left(x_i,r_i/10\right)}\left|\tilde{c}\right|^p\dd\mu,
    \end{multline*}
    where \(\tilde{c}:=\frac{10^{\kappa \alpha i}}{r_i^{1-\alpha}}\left(r_i/10-d\left(x,x_i\right)-r\right)\). For \(y\in B\left(x,r\right)\cap B\left(x_i,r_i/10\right)\) we have 
    \begin{align*}
        \left|f\left(y\right)-\tilde{c}\right|^p&=\left|\frac{10^{\kappa \alpha i}}{r_i^{1-\alpha}}\left(r_i/10-d\left(y,x_i\right)-r_i/10+d\left(x,x_i\right)+r\right)\right|^p\\
        &\leq \frac{10^{\kappa \alpha ip}}{r_i^{p-\alpha p}} \left(\left|d\left(x,x_i\right)-d\left(y,x_i\right)\right|+r\right)^p\\
        &\leq \frac{10^{\kappa \alpha ip}}{r_i^{p-\alpha p}} \left(d\left(x,y\right)+r\right)^p\leq \frac{10^{\kappa \alpha ip}}{r_i^{p-\alpha p}} \left(2r\right)^p.
    \end{align*}
    If there exists \(y\in B\left(x,r\right)\setminus B\left(x_i,r_i/10\right)\), then
     \[r_i/10\leq d\left(x_i,y\right)\leq d\left(x_i,x\right)+d\left(x,y\right)<d\left(x,x_i\right)+r.\]
     Moreover, from condition \eqref{eq:BiCondition} we get that there exists \(z\in B\left(x,r\right)\cap B\left(x_i,r_i/10\right)\), so 
     \[d\left(x,x_i\right)\leq d\left(x,z\right)+d\left(z,x_i\right)<r+r_i/10.\]
     Hence, if \(B\left(x,r\right)\setminus B\left(x_i,r_i/10\right)\neq\emptyset\), then
    \begin{align*}
        \left|\tilde{c}\right|^p=\left|\frac{10^{\kappa \alpha i}}{r_i^{1-\alpha}}\left(r_i/10-d\left(x,x_i\right)-r\right)\right|^p=\frac{10^{\kappa \alpha ip}}{r_i^{p-\alpha p}}\left(r+d\left(x,x_i\right)-r_i/10\right)^p<\frac{10^{\kappa \alpha ip}}{r_i^{p-\alpha p}}\left(2r\right)^p.
    \end{align*}
    Moreover, from condition \eqref{eq:BiCondition} we get that 
    \[B\left(x,r\right)\subseteq B\left(x_i, r_i\right).\]
    Indeed, let \(y\in B\left(x,r\right)\) and \(z\in B\left(x,r\right)\cap B\left(x_i,r_i/10\right)\), then
    \[d\left(x_i,y\right)\leq d\left(x_i,z\right)+d\left(z,x\right)+d\left(x,y\right)<r_i/10+r+r\leq r_i/10+r_i/10+r_i/10<r_i.\]

If $\hbox{I})$ holds, then we obtain
    \begin{eqnarray*}
        \int_{B\left(x,r\right)\cap B\left(x_i,r_i/10\right)}&&\left|f-\tilde{c}\right|^p\dd\mu+\int_{B\left(x,r\right)\setminus B\left(x_i,r_i/10\right)}\left|\tilde{c}\right|^p\dd\mu \\   &\leq&\int_{B\left(x,r\right)\cap B\left(x_i,r_i/10\right)}\frac{10^{\kappa \alpha ip}}{r_i^{p-\alpha p}}\left(2r\right)^p\dd\mu+\int_{B\left(x,r\right)\setminus B\left(x_i,r_i/10\right)}\frac{10^{\kappa \alpha ip}}{r_i^{p-\alpha p}}\left(2r\right)^p\dd\mu\\
        &=&\int_{B\left(x,r\right)}\frac{10^{\kappa \alpha ip}}{r_i^{p-\alpha p}}\left(2r\right)^p\dd\mu=\frac{10^{\kappa \alpha ip}}{r_i^{p-\alpha p}} 2^pr^p\mu\left(B\left(x,r\right)\right)\\
&\leq&\frac{10^{\kappa \alpha ip}}{r_i^{p-\alpha p}} 2^pr^p\mu\left(B\left(x_i,r_i\right)\right)
        \leq C\frac{10^{\left(\kappa \alpha p-1\right)i}}{r_i^{p-\alpha p}} 2^pr^pr_i^Q.
    \end{eqnarray*}
    From the above and inequality \eqref{eq:BxrEstimate1}, we get 
    \begin{align*}
       \frac{1}{r^{\lambda}} \inf_{c\in\mathbb{R}} \int_{B\left(x,r\right)}\left|f-c\right|^p\dd\mu &\leq C2^p\frac{1}{r^{\lambda}}\min\left\{r_i^{\lambda}10^{\left(\kappa\alpha p-1\right)i}, \frac{10^{\left(\kappa\alpha p-1\right)i}}{r_i^{p-\alpha p}} r^pr_i^Q\right\}\\
         &=C2^pr_i^{\lambda}10^{\left(\kappa\alpha p-1\right)i}\min\left\{r^{-\lambda}, r_i^{-p} r^{p-\lambda}\right\}.
    \end{align*}
     Let us observe that the map  \(r\mapsto \min\left\{r^{-\lambda}, r_i^{-p} r^{p-\lambda}\right\}\) has maximum value at $r=r_i$. Thus, we get
    \begin{align*}
       \frac{1}{r^{\lambda}} \inf_{c\in\mathbb{R}} \int_{B\left(x,r\right)}\left|f-c\right|^p\dd\mu\leq C2^p10^{\left(\kappa\alpha p-1\right)i} \leq C2^p10^{-i/2}<C2^p,
    \end{align*}
since $\kappa \leq (2 \alpha p)^{-1}$.

If $\hbox{II})$ holds, we obtain
    \[
        \int_{B\left(x,r\right)\cap B\left(x_i,r_i/10\right)}\left|f-\tilde{c}\right|^p\dd\mu+\int_{B\left(x,r\right)\setminus B\left(x_i,r_i/10\right)}\left|\tilde{c}\right|^p\dd\mu \leq \frac{10^{\kappa \alpha ip}}{r_i^{p-\alpha p}} 2^pr^p C_Qr^Q.
    \]
    Gathering \eqref{eq:BxrEstimate1} with the above inequality we get 
    \begin{align*}
       \frac{1}{r^{\lambda}} \inf_{c\in\mathbb{R}} \int_{B\left(x,r\right)}\left|f-c\right|^p\dd\mu &\leq \Tilde{C}\frac{1}{r^{\lambda}}\min\left\{r_i^{\lambda}10^{\left(\kappa\alpha p-1\right)i}, \frac{10^{\kappa \alpha ip}}{r_i^{p-\alpha p}} r^{p+Q}\right\}\\
         &=\Tilde{C}r_i^{\alpha p}10^{\kappa\alpha pi}\min\left\{r_i^{Q}10^{-i}r^{-\lambda}, r_i^{-p} r^{p\left(1-\alpha\right)}\right\},
    \end{align*}
    where \(\Tilde{C}=\max\left\{C, 2^pC_Q\right\}\). Since the map  \(r\mapsto \min\left\{r_i^{Q}10^{-i}r^{-\lambda}, r_i^{-p} r^{p\left(1-\alpha\right)}\right\}\) has maximum value at $r=r_i10^{-\frac{i}{p+Q}}$ and \(\kappa\leq\frac{1-\alpha}{\alpha\left(p+Q\right)}\), we get
    \begin{align*}
       \frac{1}{r^{\lambda}} \inf_{c\in\mathbb{R}} \int_{B\left(x,r\right)}\left|f-c\right|^p\dd\mu\leq\Tilde{C}r_i^{\alpha p}10^{\kappa\alpha pi}r_i^{-p}\left(r_i10^{-\frac{i}{p+Q}}\right)^{p\left(1-\alpha\right)}=\Tilde{C}10^{ip\left(\kappa\alpha-\frac{1-\alpha}{p+Q}\right)}\leq \Tilde{C}.
    \end{align*}

    Now let us come back to inequality \eqref{eq:notBiNorm}. Denote by \(k\in\mathbb{Z}_+\) the smallest index such that \(r_k/10\leq r\), then from inequality \eqref{eq:BnEstimate} we get that
    \begin{align*}
       \frac{1}{r^{\lambda}} \inf_{c\in\mathbb{R}}\int_{B\left(x,r\right)}\left|f-c\right|^p\dd\mu&\leq\frac{1}{r^\lambda}\sum_{\left\{j\in\mathbb{Z}_+:r_j/10\leq r\right\}}\int_{B\left(x_j,r_j/10\right)}\left|f\right|^p\dd\mu\leq \frac{1}{r^\lambda}\sum_{j\geq k} Cr_j^{\lambda}10^{\left(\kappa\alpha p-1\right)j}\\
        &\leq C\frac{r_k^{\lambda}}{r^{\lambda}}\frac{10^{\left(\kappa\alpha p-1\right)k}}{1-10^{\left(\kappa\alpha p-1\right)}}\leq C\frac{10^{\lambda}}{1-10^{\left(\kappa\alpha p-1\right)}}.
    \end{align*} 

Finally, let us recall that
 \begin{align*}
        \vertiii{f}_{p,\lambda}=\left(\sup_{x\in X, r>0}\frac{1}{r^{\lambda}}\inf_{c\in\mathbb{R}}\int_{B(x,r)}\left|f-c\right|^p\dd\mu\right)^{\frac{1}{p}}.
    \end{align*}
Therefore, gathering both cases we obtain
        \begin{align*}
        \vertiii{f}_{p,\lambda}\leq \begin{cases}
             \max\left\{C2^p,C \frac{10^{\lambda}}{1-10^{\kappa \alpha p-1}} \right\}^{\frac{1}{p}} & \text{if  \hbox{I}) holds};\\
             \max\left\{\tilde{C},C \frac{10^{\lambda}}{1-10^{\kappa \alpha p-1}} \right\}^{\frac{1}{p}} &  \text{if  \hbox{II}) holds}.
            \end{cases}
\end{align*} 
 and since $f\in L^{p}(X)$, we conclude that \(f\in L^{p,\lambda}_C\left(X,\mu\right)\).
    \end{proof}
 Gathering Theorem \ref{camp} with Theorem \ref{dolne} we get.
 \begin{tw} \label{ost1}
        Let \(\left(X,d,\mu\right)\) be a metric-measure space such that \(\left(X,d\right)\) is compact and \(\eta\)-uniformly perfect. Let $Q \in (0, \infty), p \in [1, \infty)$, and  $\lambda\in\left(Q,Q+p\right)$. If the measure \(\mu\) is Borel regular and upper \(Q\)-regular, then the following conditions are equivalent.
\begin{enumerate}
\item[i)] The measure $\mu$ is lower $Q$-regular;
\item[ii)] $L^{p,\lambda}_C\left(X,\mu \right)\xhookrightarrow{}C^{0,\alpha}\left(X\right),$
where $\alpha = (\lambda-Q)/p.$
\end{enumerate}
    \end{tw}
\begin{tw} \label{ost2}
        Let \(\left(X,d,\mu\right)\) be a metric-measure space such that \(\left(X,d\right)\) is compact and \(\eta\)-uniformly perfect. Let $Q \in (0, \infty), p \in [1, \infty)$ and  $\lambda\in\left(Q,p\right]$. If the measure \(\mu\) is doubling and Borel regular, then the following conditions are equivalent.
\begin{enumerate}
\item[i)] The measure $\mu$ is lower $Q$-regular;
\item[ii)] $L^{p,\lambda}_C\left(X,\mu \right)\xhookrightarrow{}C^{0,\alpha}\left(X\right),$
where $\alpha = (\lambda-Q)/p.$
\end{enumerate}
    \end{tw}
\begin{q}
Is it true that Theorem \ref{ost2} holds for $\lambda\in\left(Q,Q+p\right)$?
\end{q}

\textbf{Acknowledgements:} We would like to cordially thank Paul Howard for discussion related to Example \ref{Paul}. We would also like to thank David Op\v{e}la for sending us his dissertation.

\smallskip
{\small Przemys{\l}aw  G\'orka}\\
\small{Faculty of Mathematics and Information Science,}\\
\small{Warsaw University of Technology,}\\
\small{Pl. Politechniki 1, 00-661 Warsaw, Poland} \\
{\tt przemyslaw.gorka@pw.edu.pl}\\
\\
{\small Zuzanna  Stępie\'n}\\
\small{Faculty of Mathematics and Information Science,}\\
\small{Warsaw University of Technology,}\\
\small{Pl. Politechniki 1, 00-661 Warsaw, Poland} \\
{\tt zuzanna.stepien2.dokt@pw.edu.pl}\\

\end{document}